\documentclass[11pt,english]{smfart}

\usepackage[T1]{fontenc}
\usepackage[english,francais]{babel}
\usepackage[all,cmtip]{xy}
\usepackage{amssymb,url,xspace,smfthm}
\usepackage{amsmath}
\usepackage{color}
\usepackage{float}
\usepackage{caption}
\usepackage{lmodern}

\usepackage{microtype}
\usepackage{framed}
\usepackage{tikz}
\usepackage{graphicx}
\usepackage{wasysym}

\makeatother

\newcommand{\BibTeX}{{\scshape Bib}\kern-.08em\TeX}
\newcommand{\T}{\S\kern .15em\relax }
\newcommand{\AMS}{$\mathcal{A}$\kern-.1667em\lower.5ex\hbox
        {$\mathcal{M}$}\kern-.125em$\mathcal{S}$}

\author{Pedro \textsc{Montero}}
\address{Univ. Grenoble Alpes, Institut Fourier, F-38000 Grenoble, France.}
\email{pedro.montero@univ-grenoble-alpes.fr}
\title{Newton-Okounkov bodies on projective bundles over curves}

\newcommand{\Psef}{\operatorname{Psef}}
\newcommand{\Nef}{\operatorname{Nef}}
\newcommand{\Amp}{\operatorname{Amp}}
\newcommand{\Bg}{\operatorname{Big}}

\newcommand{\NE}{\operatorname{NE}}
\newcommand{\HN}{\operatorname{HN}}
\newcommand{\N}{\operatorname{N}}

\newcommand{\codim}{\operatorname{codim}}

\newcommand{\ord}{\operatorname{ord}}

\newcommand{\h}{\operatorname{H}}
\newcommand{\rk}{\operatorname{rank}}

\newcommand{\vol}{\operatorname{vol}}
\newcommand{\bs}{\operatorname{Bs}}

\newtheorem{thm}{Theorem}[section]
\newtheorem*{thm*}{Theorem}
\newtheorem*{thma}{Theorem A}
\newtheorem*{propb}{Proposition B}

\newtheorem*{mainthm*}{Main Theorem}

\newtheorem{lemma}[thm]{Lemma}
\newtheorem{cor}[thm]{Corollary}

\newtheorem{propo}[thm]{Proposition}

\newtheorem{remark}[thm]{Remark}

\theoremstyle{definition}
\newtheorem{defn}[thm]{Definition}

\newtheorem{exmp}[thm]{Example}

\newtheorem{defn-thm}[thm]{Definition-Theorem} 
\newtheorem{convention}[thm]{Convention}

\theoremstyle{remark}
\newtheorem{notation}[thm]{Notation}

\newtheorem*{not-and-def}{Notation and definitions}

\numberwithin{equation}{section}

\usepackage[square,sort,comma,numbers]{natbib}
\begin{document}

\def\smfbyname{}

\begin{abstract}
In this article, we study Newton-Okounkov bodies on projective vector bundles over curves. Inspired by Wolfe's estimates used to compute the volume function on these varieties, we compute all Newton-Okounkov bodies with respect to linear flags. Moreover, we characterize semi-stable vector bundles over curves via Newton-Okounkov bodies.
\end{abstract}

\maketitle

\tableofcontents
\section{Introduction}

Let $C$ be a smooth projective curve over an algebraically closed field of characteristic zero and let $E$ be vector bundle over $C$ of rank $r\geq 2$. It is well-known since Hartshorne's work \cite{Ha71} that numerical information coming from semi-stability properties of $E$ can be translated into positivity conditions. Namely, the shape of the nef and pseudo-effective cone of the projective vector bundle $\mathbb{P}(E)$ were determined by Hartshorne \cite{Ha71} (cf. \cite[\S 6.4.B]{Laz04}) and Nakayama \cite[Chapter IV]{Nak04} (cf. \cite{MDS15}) in terms of the Harder-Narasimhan filtration of $E$. More precisely, if we denote by $\xi$ the class of the tautological line bundle $\mathcal{O}_{\mathbb{P}(E)}(1)$ and by $f$ the class of a fiber of $\pi:\mathbb{P}(E)\to C$ then we have that for $t\in \mathbb{R}$ the class $\xi - tf$ is nef (resp. pseudo-effective) if and only if $t\leq \mu_{\min}(E)$ (resp. $t\leq \mu_{\max}(E)$), where $\mu_{\min}(E)$ (resp. $\mu_{\max}(E)$) is the minimal (resp. maximal) slope of $E$. In particular, the nef and pseudo-effective cone of $\mathbb{P}(E)$ coincide if and only if $E$ is semi-stable (cf. \cite[Theo. 3.1]{Miy87} and \cite{Ful11}).

Indirectly, the pseudo-effective cone can be also deduced from the work of Wolfe \cite{Wol05} and Chen \cite{Che08}, who explicitly computed the volume function on $\mathbb{P}(E)$. In fact, they showed that for every $t\in \mathbb{R}$ the volume of the class $\xi - tf$ on $\mathbb{P}(E)$ can be expressed in terms of numerical information coming from the Harder-Narasimhan filtration of $E$,
$$\operatorname{HN}_\bullet(E) : 0=E_\ell \subseteq E_{\ell-1} \subseteq \cdots \subseteq E_1 \subseteq E_0=E $$
with successive semi-stable quotients $Q_i=E_{i-1}/E_i$ of rank $r_i$ and slope $\mu_i$. More precisely, if we consider $\sigma_1 \geq \ldots \geq \sigma_r$ to be the ordered slopes of $E$ counted with multiplicities equal to the rank of the corresponding semi-stable quotient\footnote{In other words, $\mu_{\max}(E)=\sigma_1 \geq \ldots \geq \sigma_r=\mu_{\min}(E)$ can be viewed as the components of the vector $\boldsymbol{\sigma}=(\sigma_1,\ldots,\sigma_r)=(\underbrace{\mu_\ell,\ldots,\mu_\ell}_{r_\ell\text{ times}},\underbrace{\mu_{\ell-1},\ldots,\mu_{\ell-1}}_{r_{\ell-1} \text{ times}},\ldots,\underbrace{\mu_1,\ldots,\mu_1}_{r_1 \text{ times}})\in \mathbb{Q}^r$.}, their results can be summarized as follows.

\begin{theo}\label{theo: volume result, wolfe and chen}
Let $C$ and $E$ as above. Then,
$$\vol_{\mathbb{P}(E)}(\xi-tf)=r! \cdot \int_{\widehat{\Delta}_{r-1}} \max \left\{\sum_{j=1}^r \sigma_{r+1-j} \lambda_j - t,0 \right\}\; d\lambda, $$
\noindent  where ${\widehat{\Delta}}_{r-1}\subseteq \mathbb{R}^r$ is the standard $(r-1)$-simplex with coordinates $\lambda_1,\ldots,\lambda_r$ and $d\lambda$ is the standard induced Lebesgue measure for which $\widehat{\Delta}_{r-1}$ has volume $\frac{1}{(r-1)!}$.
\end{theo}

Following the idea that numerical information encoded by the Harder-Narasimhan filtration of $E$ should be related to asymptotic numerical invariants of $\mathbb{P}(E)$, we study the geometry of Newton-Okounkov bodies on $\mathbb{P}(E)$. These compact convex bodies were introduced by Okounkov in his original article \cite{Ok96} and they were studied later on by Kaveh and Khovanskii \cite{KK12} and Lazarsfeld and Musta\c{t}\u{a} \cite{LM09}, who associated to any big divisor $D$ on a normal projective variety $X$ of dimension $r$, and any complete flag of subvarieties $Y_\bullet$ on $X$ satisfying suitable conditions, a convex body $\Delta_{Y_\bullet}(D)\subseteq \mathbb{R}^r$ depending only on the numerical equivalence class of $D$. Moreover, there exists a {\it global Newton-Okounkov body} $\Delta_{Y_\bullet}(X)\subseteq \mathbb{R}^r\times \N^1(X)_{\mathbb{R}}$ such that the slice of $\Delta_{Y_\bullet}(X)$ over any big rational class $\eta\in \N^1(X)_{\mathbb{Q}}$ is given by $\Delta_{Y_\bullet}(\eta)\subseteq \mathbb{R}^r\times \{\eta\}$.

Newton-Okounkov bodies of big divisors on geometrically ruled surfaces with respect to {\it linear flags} (see Definition \ref{defi: natural flag}) can be computed via Zariski decomposition (see Example \ref{ex:ruled surfaces}). In higher dimension, we will use methods similar to those used by Wolfe to compute the volume function in \cite{Wol05}. More precisely, we will first reduce ourselves to the case of the nef and big class $\xi-\mu_{\min}(E)f$ (see Lemma \ref{lemma:reduction to big nef}). Afterwards, we will need to understand the Harder-Narasimhan filtration of the symmetric products $S^m E$ for $m\geq 1$ and then to consider suitable refinements of these filtrations. 

Let $Y_\bullet$ be a complete {\it linear flag which is compatible with the Harder-Narasimhan filtration of $E$} (see Definition \ref{defi: compatible natural flags}). 
With the notation of Theorem \ref{theo: volume result, wolfe and chen} above, define for each real number $t\in \mathbb{R}$ the following polytope inside the full dimensional standard simplex $\Delta_{r-1}$ in $\mathbb{R}^{r-1}$ (see Notation \ref{notation: Wolfe's square})
$$\square_t=\left\{(\nu_2,\ldots,\nu_r)\in \Delta_{r-1}\;\left|\; \sum_{i=2}^r \sigma_{i-1}\nu_i+\sigma_{r}\left(1-\sum_{i=2}^r \nu_i\right) \geq t  \right.\right\}.$$

\noindent Then, we prove the following result (see Corollary \ref{cor: main result}).

\begin{thma}\label{thm: main theorem}
Let $C$ be a smooth projective curve and let $E$ be a vector bundle over $C$ of rank $r\geq 2$. Then, for every real number $t<\mu_{\max}(E)=\mu_\ell$ we have that
$$\Delta_{Y_\bullet}(\xi-t f)=\left\{(\nu_1,\ldots,\nu_r)\in \mathbb{R}^r_{\geq 0}\;|\;0 \leq \nu_1 \leq \mu_\ell - t,\;(\nu_2,\ldots,\nu_r)\in \square_{t+\nu_1} \right\}. $$
In particular, the global Newton-Okounkov body $\Delta_{Y_\bullet}(\mathbb{P}(E))$ is a rational polyhedral cone and it depends only on $\operatorname{gr}(\operatorname{HN}_\bullet(E))$, the graded vector bundle associated to the Harder-Narasimhan filtration of $E$.
\end{thma}

Moreover, we obtain the following characterization of semi-stability in terms of Newton-Okounkov bodies.

\begin{propb}\label{propo:semi-stable case}
Let $C$ be a smooth projective curve and let $E$ be a vector bundle over $C$ of rank $r\geq 2$. The following conditions are then equivalent:
 \begin{itemize}
  \item[(1)] $E$ is semi-stable.
  \item[(2)] For every big rational class $\eta=a(\xi-\mu_\ell f)+bf$ on $\mathbb{P}(E)$ and every linear flag $Y_\bullet$ on $\mathbb{P}(E)$ we have that $\Delta_{Y_\bullet}(\eta)=[0,b]\times a\Delta_{r-1} \subseteq \mathbb{R}^r$.
 \end{itemize}
Here, $a\Delta_{r-1} =\{(\nu_2,\ldots,\nu_r)\in \mathbb{R}_{\geq 0}^{r-1}\;|\;\sum_{i=2}^r \nu_i \leq a \}$ is the full dimensional standard $(r-1)$-simplex with side length $a$.
\end{propb}

\subsubsection*{{\bf Outline of the article}} First of all, we establish some notation and recall some basic facts in $\S 2$. Secondly, we recall in $\S 3$ definitions and well-known results about Newton-Okounkov bodies and semi-stability of vector bundles over curves. Next, $\S 4$ is devoted to the different cones of divisors on $\mathbb{P}(E)$ as well as results concerning their volume and restricted volume. Finally, we prove both Theorem A and Proposition B in $\S 5$.

\subsubsection*{{\bf Acknowledgements}} I would like to express my gratitude to my thesis supervisors, St\'ephane \textsc{Druel} and Catriona \textsc{Maclean}, for their advice, helpful discussions and encouragement throughout the preparation of this article. I also thank Bruno \textsc{Laurent}, Laurent \textsc{Manivel} and Bonala \textsc{Narasimha Chary} for fruitful discussions. Finally, I would like to thank the anonymous referee for a very helpful and detailed report.

\section{{Preliminaries}}

Throughout this article all varieties will be assumed to be reduced and {irreducible} schemes of finite type over a fixed algebraically closed field of characteristic zero\footnote{In positive characteristic one should take into account iterates of the absolute Frobenius as in \cite{BP14,BHP14}. It is worth mentioning that in positive characteristic the semistability of a vector bundle over a curve does not imply the semi-stability of its symmetric powers.} $k$.

\subsection{Numerical classes and positivity} We denote by $\N^1(X)$ the group of numerical equivalence classes of Cartier divisors on $X$, and we define $\N^1(X)_{\Bbbk}=\N^1(X)\otimes_{\mathbb{Z}}\Bbbk$ for $\Bbbk=\mathbb{Q}$ or $\mathbb{R}$. All the $\mathbb{R}$-divisors that we consider are $\mathbb{R}$-Cartier. Dually, we denote by $\N_1(X)$ the group of numerical equivalence classes of 1-cycles on $X$. Inside $\N_1(X)_{\mathbb{R}}=\N_1(X)\otimes_\mathbb{Z}\mathbb{R}$ we distinguish the {\it Mori cone} $\overline{\NE}(X)\subseteq \N_1(X)_\mathbb{R}$, which is the closed and convex cone generated by numerical classes of 1-cycles with non-negative real coefficients.

Let $E$ be a locally free sheaf on a variety $X$. We follow Grothendieck's convention and we define the {\it projectivization $\mathbb{P}_X(E)=\mathbb{P}(E)$ of $E$} to be ${\mathbf{Proj}_{\mathcal{O}_X} \oplus_{m\geq 0} S^m E}$, where $S^m E$ denotes de $m$th symmetric power of $E$. This variety is endowed with a natural projection $\pi:\mathbb{P}(E)\to X$ and a tautological line bundle $\mathcal{O}_{\mathbb{P}(E)}(1)$.

Let $X$ be a normal projective variety. Following \cite{Laz04}, we say that a numerical class $\eta\in \N^1(X)_{\mathbb{R}}$ is {\it big} if there exists an effective $\mathbb{R}$-divisor $E$ such that $\eta-E$ is ample. We denote by $\Bg(X)\subseteq \N^1(X)_{\mathbb{R}}$ the open convex cone of big numerical classes. A numerical class $\eta\in \N^1(X)_{\mathbb{R}}$ is called {\it pseudo-effective} if it can be written as the limit of classes of effective $\mathbb{R}$-divisors. The pseudo-effective cone is the closure of the big cone: $\overline{\Bg}(X)=\Psef(X)$ (see \cite[Theo. 2.2.26]{Laz04}, for instance). Moreover, a numerical class $\eta\in \N^1(X)_{\mathbb{R}}$ is {\it nef} if $\eta\cdot [C]\geq 0$ for every $[C]\in \overline{\NE}(X)$, and is {\it ample} if it is the numerical class of an $\mathbb{R}$-divisor that can be written as a finite sum of ample Cartier divisors with positive real coefficients. The cone $\Nef(X)\subseteq \N^1(X)_{\mathbb{R}}$ of nef classes is closed convex, and its interior $\Amp(X)$ is the cone of ample classes, by Kleiman's ampleness criterion. A line bundle $L$ is big (resp. pseudo-effective, ample, nef) if and only if its numerical class $c_1(L)\in \N^1(X)_{\mathbb{R}}$ is big (resp. pseudo-effective, ample, nef).

Let us recall that the {\it stable base locus} of a $\mathbb{Q}$-divisor $D$ on $X$ is the closed set
$$\mathbf{B}(D)=\bigcap_{m>0}\bs(mD) $$
where $\bs(mD)$ is the base locus of the complete linear series $|mD|$. Following \cite{ELMNP06}, we define the {\it augmented base locus} of $D$ to be the Zariski closed set
$$\mathbf{B}_+(D)=\bigcap_{A}\mathbf{B}(D-A),  $$
where the intersection runs over all ample $\mathbb{Q}$-divisors $A$. Similarly, the {\it restricted base locus} of $D$ is defined by
$$\mathbf{B}_-(D)=\bigcup_{A}\mathbf{B}(D+A),  $$
where the union runs over all ample $\mathbb{Q}$-divisors $A$. The restricted base locus $\mathbf{B}_-(D)$ consist of at most a countable union of subvarieties whose Zariski closure is contained in $\mathbf{B}_+(D)$ (see \cite[Rema. 1.13]{ELMNP06} and \cite{Les14}). By \cite[Prop. 1.4, Exam. 1.8, Prop. 1.15, Exam. 1.16]{ELMNP06}, both $\mathbf{B}_{-}(D)$ and $\mathbf{B}_+(D)$ depend only on the numerical class of $D$, there is an inclusion $\mathbf{B}_{-}(D) \subseteq \mathbf{B}_+(D)$, and for any rational number $c>0$ we have ${\mathbf{B}_{-}(cD)=\mathbf{B}_{-}(D)}$ and ${\mathbf{B}_+(cD)=\mathbf{B}_+(D)}$. Moreover by \cite[Exam. 1.7, Exam. 1.18]{ELMNP06} we have that ${\mathbf{B}_+(D)=\emptyset}$ if and only if $D$ is ample, and that $\mathbf{B}_-(D)=\emptyset$ if and only if $D$ is nef.

\subsection{Flag varieties and Schubert cells} Let us denote by $\mathbb{F}_r$ the full flag variety parametrizing all complete linear flags on $\mathbb{P}^{r-1}$. Recall that if we fix a reference complete linear flag $Y_\bullet$ on $\mathbb{P}^{r-1}$, then there is a decomposition of $\mathbb{F}_r$ into {\it Schubert cells}
$$\mathbb{F}_r=\coprod_{w\in \mathfrak{S}_r}\Omega_w. $$
More explicitly, if we consider homogeneous coordinates $[x_1:\ldots:x_r]$ on $\mathbb{P}^{r-1}$, we assume that for every $i=1,\ldots,r-1$ we have
$$Y_i=\{x_1=\ldots=x_i=0\}\subseteq \mathbb{P}^{r-1}$$
and we regard the permutation group $\mathfrak{S}_r$ as a subgroup of $\operatorname{PGL}_r(k)$ via its natural action on the standard basis points $e_1,\ldots,e_r\in \mathbb{P}^{r-1}$ then we have that $\Omega_w$ is the orbit $B\cdot Y_\bullet^w$, where $B$ denotes the (Borel) subgroup of $\operatorname{PGL}_r(k)$ that fixes the flag $Y_\bullet$ and $Y_\bullet^w$ is the complete linear flag such that for every $i=1,\ldots,r-1$ we have
$$Y_i^w = \{x_{w(1)}=\ldots=x_{w(i)}=0 \} \subseteq \mathbb{P}^{r-1}.$$
We refer the reader to \cite[$\S 1.2$]{Bri05} and \cite[$\S 3.6$]{Man98} for further details.

\section{{Newton-Okounkov bodies and Semi-stability}}

In this section, we review the construction of Newton-Okounkov bodies and semi-stability of vector bundles over smooth projective curves. 

\subsection{Newton-Okounkov bodies}

Let $X$ be a smooth projective variety of dimension $n$ and let $L$ be a big line bundle on $X$. A full flag of closed subvarieties of $X$ centered at the point $p\in X$ 
$$Y_\bullet: X = Y_0 \supseteq Y_1 \supseteq Y_2 \supseteq \cdots \supseteq Y_{n-1} \supseteq Y_n = \{p\} $$

\noindent is an {\it admissible flag} if $\codim_X(Y_i)=i$, and each $Y_i$ is smooth at the point $p$. In particular, $Y_{i+1}$ defines a Cartier divisor on $Y_i$ in a neighborhood of the point $p$. Following the work of Okounkov \cite{Ok96,Ok00}, Kaveh and Khovanskii \cite{KK12} and Lazarsfeld and Musta\c{t}\u{a} \cite{LM09} independently associated to $L$ and $Y_\bullet$ a convex body $\Delta_{Y_\bullet}(X,L)\subseteq \mathbb{R}^n$ encoding the asymptotic properties of the complete linear series $|L^{\otimes m}|$. We will follow the presentation of \cite{LM09} and we refer the interested reader to the survey \cite{Bo12} for a comparison of both points of view.

Let $D$ be any divisor on $X$ and let $s=s_1\in H^0(X,\mathcal{O}_X(D))$ be a non-zero section. We shall compute successive vanishing orders of global sections in the following manner: let $D_1=D+\operatorname{div}(s_1)$ be the effective divisor in the linear series $|D|$ defined by $s_1$ and set $\nu_1(s)=\ord_{Y_1}(D_1)$ the coefficient of $Y_1$ in $D_1$. Then $D_1 - \nu_1(s)Y_1$ is an effective divisor in the linear series $|D-\nu_1(s)Y_1|$, and does not contain $Y_1$ in its support, so we can define $D_2=(D_1-\nu_1(s)Y_1)|_{Y_1}$ and set $\nu_2(s)=\ord_{Y_2}(D_2)$. We proceed inductively in order to get $\nu_{Y_\bullet}(s)=(\nu_1(s),\ldots,\nu_n(s))\in \mathbb{N}^d$. This construction leads to a valuation-like function 
$$\nu_{Y_\bullet}: \h^0(X,\mathcal{O}_X(D))\setminus \{0\} \to \mathbb{Z}^n,\;s\mapsto (\nu_1(s),\ldots,\nu_n(s)). $$
We then define the {\it graded semigroup} of $D$ to be the sub-semigroup of $\mathbb{N}^n\times \mathbb{N}$ defined by
$$\Gamma_{Y_\bullet}(D)=\left\{(\nu_{Y_\bullet}(s),m)\in \mathbb{N}^n\times \mathbb{N}\;|\;0\neq s \in \h^0(X,\mathcal{O}_X(mD) \right\}. $$
Finally, we define the {\it Newton-Okounkov body of $D$ with respect to $Y_\bullet$} to be 
$$\Delta_{Y_\bullet}(D)=\operatorname{cone}(\Gamma_{Y_\bullet}(D))\cap (\mathbb{R}^n\times \{1\}), $$
where $\operatorname{cone}(\Gamma_{Y_\bullet}(D))$ denotes the closed convex cone in $\mathbb{R}^n \times \mathbb{R}$ spanned by $\Gamma_{Y_\bullet}(D)$.

These convex sets $\Delta_{Y_\bullet}(D)$ are compact and they have non-empty interior whenever $D$ is big. Moreover, by \cite[Theo. A]{LM09}, we have the following identity
$$\vol_{\mathbb{R}^n}(\Delta_{Y_\bullet}(D))=\dfrac{1}{n!}\cdot \vol_X(D), $$
where $\vol_X(D)=\lim_{m\to \infty}\frac{h^0(X,\mathcal{O}_X(mD))}{m^n/n!}$. In particular, if $D$ is big and nef, then $\vol_{\mathbb{R}^n}(\Delta_{Y_\bullet}(D))=\frac{1}{n!}D^n$, by the Asymptotic Riemann-Roch theorem.

The Newton-Okounkov bodies of big divisors depend only on numerical classes: if $D\equiv_{\text{num}} D'$ are big divisors then $\Delta_{Y_\bullet}(D)=\Delta_{Y_\bullet}(D')$ for every admissible flag $Y_\bullet$ on $X$, by \cite[Prop. 4.1]{LM09} (see \cite[Theo. A]{Jow10} for the converse). This fact, along with the identity $\Delta_{Y_\bullet}(pD)=p\cdot \Delta_{Y_\bullet}(D)$ for every positive integer $p$, enables us to define an Newton-Okounkov body $\Delta_{Y_\bullet}(\eta)\subseteq \mathbb{R}^n$ for every big rational class $\eta\in \Bg(X)\cap \N^1(X)_\mathbb{Q}$. Moreover, by \cite[Theo. B]{LM09}, there exists a {\it global Newton-Okounkov body}: a closed convex cone
$$\Delta_{Y_\bullet}(X)\subseteq \mathbb{R}^n\times \N^1(X)_\mathbb{R} $$
such that for each big rational class $\eta \in \Bg(X)_\mathbb{Q}=\Bg(X)\cap \N^1(X)_\mathbb{Q}$ the fiber of the second projection over $\eta$ is $\Delta_{Y_\bullet}(\eta)$. This enables us to define Newton-Okounkov bodies for big real classes by continuity.

The above construction works for {\it graded linear series} $A_\bullet$ associated to a big divisor $D$ on $X$. A graded linear series is a collection of subspaces $A_m\subseteq \h^0(X,\mathcal{O}_X(mD))$ such that $A_{\bullet}=\oplus_{m\geq 0}A_m$ is a graded subalgebra of the section ring ${R(D)=\oplus_{m\geq 0} \h^0(X,\mathcal{O}_X(mD))}$. The construction enables us to attach to any graded linear series $A_\bullet$ a closed and convex set $\Delta_{Y_\bullet}(A_\bullet)\subseteq \mathbb{R}^n$. This set $\Delta_{Y_\bullet}(A_\bullet)$ will be compact and will compute the volume of the linear series under some mild conditions listed in \cite[$\S$ 2.3]{LM09}. We will be specially interested on {\it restricted complete linear series} of a big divisor $D$, namely graded linear series of the form
$$A_m = \h^0(X|F,\mathcal{O}_X(mD))=\text{Im}\left(\h^0(X,\mathcal{O}_X(mD))\xrightarrow{\operatorname{rest}} \h^0(F,\mathcal{O}_F(mD)) \right) $$
where $F\subseteq X$ is an irreducible subvariety of dimension $d\geq 1$. Under the hypothesis that $F\not\subseteq \mathbf{B}_+(D)$, the conditions listed in \cite[$\S$ 2.3]{LM09} are satisfied by \cite[Lemm. 2.16]{LM09}. Therefore, the Newton-Okounkov body associated to $A_\bullet$ above, the {\it restricted Newton-Okounkov body} (with respect to a fixed admissible flag)
$$\Delta_{X|F}(D)\subseteq \mathbb{R}^d, $$
is compact and 
$$\vol_{\mathbb{R}^d}(\Delta_{X|F}(D)) = \frac{1}{d!}\vol_{X|F}(D),$$
\noindent where
$$\vol_{X|F}(D)= \lim_{m\to \infty} \frac{\dim_k A_m}{m^d/d!}$$
\noindent is the {\it restricted volume} on $F$ of the divisor $D$. In particular, if $D$ is big and nef, then ${\vol_{X|F}(D)=(D^d\cdot F)}$, by \cite[Cor. 2.17]{ELMNP09}. Restricted Newton-Okounkov bodies depend only on numerical classes (see \cite[Rema. 4.25]{LM09}), so it is meaningful to consider $\Delta_{X|F}(\eta)$ for every big rational class $\eta$ such that $F\not\subseteq \mathbf{B}_+(\eta)$. 

As before, there exists a global Newton-Okounkov body $\Delta_{Y_\bullet}(X|F)$ that enables us to define, by continuity, $\Delta_{X|F}(\eta)$ for any big real numerical class $\eta$ such that $F\not\subseteq \mathbf{B}_+(\eta)$. See \cite[Exam. 4.24]{LM09} for details. 

Restricted Newton-Okounkov bodies can be used to describe slices of Newton-Okounkov bodies.

\begin{thm}[{\cite[Theo. 4.26, Cor. 4.27]{LM09}}]\label{slices} \noindent Let $X$ be a normal projective variety of dimension $n$, and let $F\subseteq X$ be an irreducible and reduced Cartier divisor on $X$. Fix an admissible flag
 
 $$Y_\bullet: X = Y_0 \supseteq Y_1 \supseteq Y_2 \supseteq \cdots \supseteq Y_{n-1} \supseteq Y_n = \{p\} $$
 
\noindent with divisorial component $Y_1=F$. Let $\eta\in \Bg(X)_\mathbb{Q}$ be a rational big class, and consider the Newton-Okounkov body $\Delta_{Y_\bullet}(\eta)\subseteq \mathbb{R}^n$. Write $\operatorname{pr}_1:\Delta_{Y_\bullet}(\eta) \to \mathbb{R}$ for the projection onto the first coordinate, and set
\begin{equation*}
 \begin{aligned}
  \Delta_{Y_\bullet}(\eta)_{\nu_1=t}&=\operatorname{pr}_1^{-1}(t)\subseteq \{t\}\times \mathbb{R}^{n-1}\\
  \Delta_{Y_\bullet}(\eta)_{\nu_1\geq t}&=\operatorname{pr}_1^{-1}\left([t,+\infty) \right)\subseteq \mathbb{R}^n  
 \end{aligned}
\end{equation*}
Assume that $F\not\subseteq \mathbf{B}_+(\eta)$ and let 
$$\tau_F(\eta)=\sup\{s>0\;|\;\eta - s\cdot f \in \Bg(X)\}, $$
\noindent where $f\in \N^1(X)$ is the numerical class of $F$. Then, for any $t\in \mathbb{R}$ with $0\leq t < \tau_F(\eta)$ we have
\begin{enumerate}
 \item $\Delta_{Y_\bullet}(\eta)_{\nu_1\geq t}=\Delta_{Y_\bullet}(\eta-tf)+t\cdot \vec{e}_1$, where $\vec{e}_1 = (1,0,\ldots,0)\in \mathbb{N}^n$ is the first standard unit vector. \footnote{In fact, this statement remains true even if we do not assume that $E\not\subseteq \mathbf{B}_+(\eta)$. See \cite[Prop. 1.6]{KL15b}.}
 \item $\Delta_{Y_\bullet}(\eta)_{\nu_1=t}=\Delta_{X|F}(\eta-tf)$.
 \item The function $t\mapsto \vol_X(\eta+tf)$ is differentiable at $t=0$, and 
 $$\dfrac{d}{dt}\left(\vol_X(\eta+tf) \right)|_{t=0}= n\cdot \vol_{X|F}(\eta). $$
\end{enumerate}

\end{thm}

We will need the following observation by K\"{u}ronya, Lozovanu and Maclean.

\begin{propo}[{\cite[Prop. 3.1]{KLM12}}]\label{propo:psef=nef} Let $X$ be a normal projective variety together with an admissible flag $Y_\bullet$. Suppose that $D$ is a big $\mathbb{Q}$-divisor such that $Y_1\not\subseteq \mathbf{B}_+(D)$ and that $D-tY_1$ is ample for some $0\leq t < \tau_{Y_1}(D)$, where $ \tau_{Y_1}(D)=\sup\{s>0\;|\;D-sY_1 \text{ is big} \}$. Then
$$\Delta_{Y_\bullet}(X,D)_{\nu_1=t}=\Delta_{Y_\bullet|Y_1}(Y_1,(D-tY_1)|_{Y_1}), $$
where $Y_\bullet|Y_1 : Y_1\supseteq Y_2 \supseteq \cdots \supseteq Y_n$ is the induced admissible flag on $Y_1$.

In particular, if $\Psef(X)=\Nef(X)$ then we have that the Newton-Okounkov body $\Delta_{Y_\bullet}(D)$ is the closure in $\mathbb{R}^n$ of the following set
$$\{(t,\nu_2,\ldots,\nu_n)\in \mathbb{R}^n\;|\;0\leq t < \tau_{Y_1}(D), \;(\nu_2,\ldots,\nu_n)\in \Delta_{Y_\bullet|Y_1}(Y_1,(D-tY_1)|_{Y_1}) \}. $$
\end{propo}

Let us finish this section with the case of Newton-Okounkov bodies on surfaces (see \cite[\S 6.2]{LM09} for details). We will use this description in Example \ref{ex:ruled surfaces} in order to illustrate the shape of Newton-Okounkov bodies on ruled surfaces.

\begin{exmp}[Surfaces]\label{ex:surfaces}
Let $S$ be a smooth projective surface together with a flag $Y_\bullet: Y_0=S \supseteq Y_1= C \supseteq Y_2=\{p\}$, where $C\subseteq S$ is a smooth curve and $p\in C$.

Let $D$ be a big $\mathbb{Q}$-divisor on $S$. Any such divisor admits a {\it Zariski decomposition}, that is we can uniquely write $D$ as a sum
$$D=P(D)+N(D) $$
\noindent of $\mathbb{Q}$-divisors, with $P(D)$ nef and $N(D)$ either zero or effective with negative definite intersection matrix. Moreover, $P(D)\cdot \Gamma = 0$ for every irreducible component $\Gamma$ of $N(D)$ and for all $m\geq 0$ there is an isomorphism $\h^0(S,\mathcal{O}_S(\lfloor mP(D) \rfloor)) \cong \h^0(S,\mathcal{O}_S(\lfloor mD \rfloor))$. In this decomposition $P(D)$ is called the {\it positive part} and $N(D)$ the {\it negative part}. See \cite[\S 2.3.E]{Laz04} and references therein for proofs and applications.

With the above notation, we have that
$$\Delta_{Y_\bullet}(D)=\left\{(t,y)\in \mathbb{R}^2\;|\; \nu \leq t \leq \tau_C(D),\; \alpha(t)\leq t \leq \beta(t) \right\} $$
where
\begin{enumerate}
 \item $\nu \in \mathbb{Q}$ the coefficient of $C$ in $N(D)$,
 \item $\tau_C(D)=\sup\{t>0\;|\;D-tC \text{ is big} \}$,
 \item $\alpha(t)=\ord_p(N_t\cdot C)$,
 \item $\beta(t)=\ord_p(N_t\cdot C)+P_t\cdot C$,
\end{enumerate}
where $D-tC=P_t+N_t$ is a Zariski decomposition, $P_t$ being the positive and $N_t$ the negative part. Moreover, these bodies are finite polygons, by \cite[Theo. B]{KLM12}.
\end{exmp}

\subsection{Semi-stability and Harder-Narasimhan filtrations}

Throughout this section, $C$ is a smooth projective curve and $E$ is a locally free sheaf on $C$ of rank $r>0$ and degree $d=\deg(E)=\deg (c_1(E))$. Given such a bundle we call the rational number
$$\mu(E)=\dfrac{d}{r} $$
the {\it slope} of $E$. 

\begin{defn}[Semi-stability]
 Let $E$ be a vector bundle on $C$ of slope $\mu$. We say that $E$ is {\it semi-stable} if for every non-zero sub-bundle $ S \subseteq E$, we have $\mu(S)\leq \mu$. Equivalently, $E$ is semi-stable if for every locally-free quotient $E \twoheadrightarrow Q$ of non-zero rank, we have $\mu \leq \mu(Q)$. A non semi-stable vector bundle will be called unstable.
\end{defn}

Following \cite[Prop. 5.4.2]{LP97}, there is a canonical filtration of $E$ with semi-stable quotients.

\begin{propo}\label{propo: HN filtration}
 Let $E$ be a vector bundle on $C$. Then $E$ has an increasing filtration by sub-bundles 
 $$ \operatorname{HN}_\bullet(E): 0=E_\ell \subseteq E_{\ell-1} \subseteq \cdots \subseteq E_1 \subseteq E_0 = E $$
 where each of the quotients $E_{i-1}/E_i$ satisfies the following conditions:
 \begin{enumerate}
  \item Each quotient $E_{i-1}/E_i$ is a semi-stable vector bundle;
  \item $\mu(E_{i-1}/E_i)<\mu(E_{i}/E_{i+1})$ for $i=1,\ldots,\ell-1$.
 \end{enumerate}
This filtration is unique.
\end{propo}

The above filtration is called the {\it Harder-Narasimhan filtration} of $E$. 

\begin{notation}\label{notation: HN quotients}
Let $E$ be a vector bundle on a smooth projective curve $C$. We will denote by $Q_i=E_{i-1}/E_i$ the semi-stable quotients of the Harder-Narasimhan filtration of $E$, each one of rank $r_i=\rk(Q_i)$, degree $d_i=\deg(c_1(Q_i))$ and slope $\mu_i=\mu(Q_i)=d_i/r_i$. With this notation, $\mu_1$ and $\mu_\ell$ correspond to the minimal and maximal slopes, $\mu_{\min}(E)$ and $\mu_{\max}(E)$, respectively.
\end{notation}

From a cohomological point of view, semi-stable vector bundles can be seen as the good higher-rank analogue of line bundles. For instance, we have the following classical properties (see \cite{RR84} or \cite[Lemm. 1.12, Lemm. 2.5]{But94}). 

\begin{lemma}\label{lemma: properties slope}
 Let $E$ and $F$ be vector bundles on $C$ and $m\in \mathbb{N}$. Then,
 \begin{enumerate}
  \item $\mu_{\max}(E\otimes F)=\mu_{\max}(E)+\mu_{\max}(F)$.
  \item $\mu_{\min}(E\otimes F)=\mu_{\min}(E)+\mu_{\min}(F)$.
  \item $\mu_{\max}(S^m E)=m \mu_{\max}(E)$.
  \item $\mu_{\min}(S^m E)=m \mu_{\min}(E)$.
  \item If $\mu_{\max}(E)<0$, then $\dim_k \h^0(C,E)=0$.
  \item If $\mu_{\min}(E)>2g-2$, then $\dim_k \h^1(C,E)=0$.
 \end{enumerate}
 In particular, if $E$ and $F$ are semi-stable then $S^m E$ and $E\otimes F$ are semi-stable.
\end{lemma}

If $E_1,\ldots, E_\ell$ are vector bundles on $C$ and $m_1,\ldots,m_\ell$ be non-negative integers. By the splitting principle \cite[Rema. 3.2.3]{Ful84} we can prove the following formula:
$$ \mu(S^{m_1}E_1 \otimes \cdots \otimes S^{m_\ell} E_\ell)=\sum_{i=1}^\ell m_i \mu(E_i).$$

Moreover, we have that for every $m\geq 1$ the Harder-Narasimhan filtration of the symmetric product $S^m E$ can be computed in terms of the one for $E$ (see \cite[Prop. 3.4]{Che08} and \cite[Prop. 5.10]{Wol05}, for instance).

\begin{propo}\label{propo: HN filtration of S^mE}
 Let $E$ be a vector bundle on $C$ with Harder-Narasimhan filtration
 $$\operatorname{HN}_\bullet(E): 0=E_\ell \subseteq E_{\ell-1} \subseteq \cdots \subseteq E_1 \subseteq E_0 = E $$
 and semi-stable quotients $Q_i=E_{i-1}/E_i$ with slopes $\mu_i=\mu(Q_i)$, for $i=1,\ldots,\ell$.
 For every positive integer $m\geq 1$, let us consider the vector bundle $S^m E$ with Harder-Narasimhan filtration
 $$\operatorname{HN}_\bullet(S^mE): 0=W_M \subseteq W_{M-1} \subseteq \cdots \subseteq W_1 \subseteq W_0 = S^m E  $$
 and semi-stable quotients $W_{j-1}/W_j$ with slopes $\nu_j=\mu(W_{j-1}/W_j)$, for ${j=1,\ldots,M}$. Then, for every $j=1,\ldots,M$ we have that 
$$W_j=\sum_{{\sum_i m_i\mu_i \geq \nu_{j+1}}}S^{m_1}E_0\otimes \cdots\otimes S^{m_\ell}E_{\ell-1} $$
and 
 $$W_{j-1}/W_j \cong \bigoplus_{\sum_i m_i \mu_i = \nu_j} S^{m_1}Q_1\otimes \cdots \otimes S^{m_\ell} Q_\ell, $$
 where the sums are taken over all partitions $\mathbf{m}=(m_1,\ldots,m_\ell)\in \mathbb{N}^\ell$ of $m$, and
 $S^{m_1}E_0\otimes \cdots \otimes S^{m_\ell}E_{\ell-1}$ denotes the image of the composite
 $$E_0^{\otimes m_1}\otimes \cdots \otimes E_{\ell-1}^{\otimes m_\ell}\to E^{\otimes m} \to S^m E. $$ 
 
In particular, there is a refinement $F_\bullet$ of $\operatorname{HN}_\bullet(S^m E)$ of length $L=L(m)$ and whose respective successive quotients are of the form 
 $$F_{i-1}/F_i\cong Q_{\mathbf{m}(i)}=S^{m_1}Q_1\otimes \cdots \otimes S^{m_\ell}Q_\ell $$
for some partition $\mathbf{m}(i)=(m_1,\ldots,m_\ell)\in \mathbb{N}^\ell$ of $m$, and such that for every $i\in \{1,\ldots L\}$ we have $\mu(Q_{\mathbf{m}(i)})\leq \mu( Q_{\mathbf{m}(i+1)})$. Moreover, given any partition $\mathbf{m}\in \mathbb{N}^\ell$ of $m$, there is one and only one $i\in \{1,\ldots,L\}$ such that $\mathbf{m}(i)=\mathbf{m}$.
\end{propo}

\section{Divisors on projective bundles over curves}

Let $E$ be a vector bundle on a smooth projective curve $C$, of rank $r\geq 2$ and degree $d$. In this section we study divisors on the projective bundle $\pi:\mathbb{P}(E)\to C$ of one-dimensional quotients.  Let us recall that in this case the N\'eron-Severi group of $\mathbb{P}(E)$ is of the form

$$\N^1(\mathbb{P}(E))=\mathbb{Z}\cdot f \oplus \mathbb{Z}\cdot \xi, $$

\noindent where $f$ is the numerical class of a fiber of $\pi$ and $\xi=\xi_E$ is the numerical class of a divisor representing the tautological line bundle $\mathcal{O}_{\mathbb{P}(E)}(1)$. Moreover, if $[\text{pt}]$ denotes the class of a point in the ring $\N^*(\mathbb{P}(E))$ then we have the following relations:
$$f^2 = 0,\; \xi^{r-1}f = [\text{pt}], \; \xi^r =  d\cdot [\text{pt}].  $$

The cone of nef divisors can be described via Hartshorne's characterization of ample vector bundles over curves \cite[Theo. 2.4]{Ha71} (cf. \cite[Lemm. 2.1]{Ful11}). 

\begin{lemma}\label{lemma: nef cone} $\Nef(\mathbb{P}(E))=\langle \xi - \mu_{\min} f,f \rangle. $
\end{lemma}

The cone of pseudo-effective divisors was obtained by Nakayama in \cite[Cor. IV.3.8]{Nak04}, and it was indirectly computed by Wolfe \cite{Wol05} and Chen \cite{Che08} who independently obtained the volume function $\vol_{\mathbb{P}(E)}$ on $\N^1(\mathbb{P}(E))_{\mathbb{R}}$. A more general result on the cone of effective cycles of arbitrary codimension can be found in \cite[Theo. 1.1]{Ful11}.

\begin{lemma}\label{lemma: psef cone}
 $\Psef(\mathbb{P}(E))=\langle \xi - \mu_{\max} f,f \rangle. $
\end{lemma}

In particular, we recover a result of Miyaoka \cite[Theo. 3.1]{Miy87} on semi-stable vector bundles over curves that was generalized by Fulger in \cite[Prop. 1.5]{Ful11}.

\begin{cor}\label{Miyaoka} A vector bundle $E$ on a smooth projective curve $C$ is semi-stable if and only if $\Nef(\mathbb{P}(E))=\Psef(\mathbb{P}(E))$. 
\end{cor}

We finish this section by recalling Wolfe's computation of the volume function on $\N^1(\mathbb{P}(E))$. See also \cite[Theo. 1.2]{Che08}.

\begin{notation}\label{notation: simplex}
 Let $d\geq 1$ be an integer. We define the {\it standard $d$-simplex} $\widehat{\Delta}_d$ with $d+1$ vertices in $\mathbb{R}^{d+1}$ to be
 $$\widehat{\Delta}_d=\Big\{(x_1,\ldots,x_{d+1})\in \mathbb{R}^{d+1}\;\Big|\; \textstyle{\sum_{i=1}^{d+1} x_i=1} \text{ and }x_i\geq 0 \text{ for all }i \Big\}. $$
By projecting $\widehat{\Delta}_d$ onto the hyperplane $x_1=0$, we can identify $\widehat{\Delta}_d$ with the {\it full dimensional standard $d$-simplex} (or just {\it $d$-simplex}) in $\mathbb{R}^d$ given by
 $$\Delta_d=\Big\{(x_2,\ldots,x_{d+1})\in \mathbb{R}^{d}\;\Big|\; \textstyle{\sum_{i=2}^{d+1} x_i\leq 1} \text{ and }x_i\geq 0 \text{ for all }i \Big\}. $$
Via the previous identification, we will denote by $\lambda$ the Lebesgue measure on $\widehat{\Delta}_d$ induced by the standard Lebesgue measure on $\Delta_d\subseteq \mathbb{R}^d$. In particular, we will have $\lambda(\widehat{\Delta}_d)=\frac{1}{d!}$.

Given a positive real number $a>0$, we define the {\it $d$-simplex with side length $a$} by $a\Delta_d=\big\{(x_2,\ldots,x_{d+1})\in \mathbb{R}^{d}\;\big|\; \sum_{i=2}^{d+1} x_i\leq a \text{ and }x_i\geq 0 \text{ for all }i \big\}$. Similar for $a \widehat{\Delta}_d \subseteq \mathbb{R}^{d+1}$, the {\it standard $d$-simplex with side length $a$}.
\end{notation}

\begin{thm}[{\cite[Theo. 5.14]{Wol05}}]\label{theo: Wolfe volume} Let $E$ be a vector bundle with Harder-Narasimhan filtration of length $\ell$ and semi-stable quotients $Q_i$ of ranks $r_i$ and slopes $\mu_i$. Then, for any $t\in \mathbb{R}$

$$\vol_{\mathbb{P}(E)}(\xi-tf)=r!\cdot \int_{\widehat{\Delta}_{\ell-1}}\max\left\{\sum_{i=1}^\ell \mu_i \beta_i - t,0\right\}\dfrac{\beta_1^{r_1-1}\cdots\beta_\ell^{r_\ell-1}}{(r_1-1)!\cdots (r_\ell-1)!} d\beta $$

\noindent where $\widehat{\Delta}_{\ell-1}\subseteq \mathbb{R}^\ell$ is the standard $(\ell-1)$-simplex with coordinates $\beta_1,\ldots,\beta_\ell$, and $\beta$ be the standard induced Lebesgue measure.
\end{thm}

\begin{remark}\label{remark: Huayi volume}
Alternatively, Chen computed in \cite[Theo. 1.2]{Che08} a similar volume formula, but slightly simplified by integrating in $\mathbb{R}^r$ instead of $\mathbb{R}^\ell$ (cf. \cite[Prop. 3.5]{Che08}). More precisely, with the same notation as above
 
 $$\vol_{\mathbb{P}(E)}(\xi-tf)= r! \cdot \int_{\widehat{\Delta}_{r-1}}\max\left\{\sum_{j=1}^r s_j \lambda_j - t,0\right\} d\lambda $$

\noindent where $\widehat{\Delta}_{r-1}\subseteq \mathbb{R}^r$ is the standard $(r-1)$-simplex with coordinates $\lambda_1,\ldots,\lambda_r$, $d\lambda$ is the standard induced Lebesgue measure\footnote{Unlike Chen, we do not normalize the measure in order to have $\lambda(\Delta_{r-1})=1$.}, and ${\mathbf{s}=(s_1,\ldots,s_r)}$ is a vector in $\mathbb{R}^r$ such that the value $\mu_i$ appears exactly $r_i$ times in the coordinates of $\mathbf{s}$ as in Notation \ref{notation: Wolfe's square} below. 
\end{remark}

\begin{notation}\label{notation: Wolfe's square}
 Fix $\ell\geq 1$ and $r\geq 1$ two integers, $(r_1,\ldots,r_\ell)\in \mathbb{N}^\ell$ a partition of $r$ and $t\in \mathbb{R}$. We define for $(\mu_1,\ldots,\mu_\ell)\in \mathbb{Q}^\ell$ the following polytopes:
 $$\widehat{\square}_{t}=\left\{(\beta_1,\ldots,\beta_{\ell})\in\widehat{\Delta}_{\ell-1}\subseteq \mathbb{R}^{\ell}\;\left|\;\sum_{i=1}^\ell \mu_i\beta_i\geq t \right.\right\} $$
 \noindent and
 $$\square_t=\left\{(\lambda_1,\ldots,\lambda_r)\in \widehat{\Delta}_{r-1}\subseteq \mathbb{R}^r\;\left|\; \sum_{i=1}^r s_i\lambda_i \geq t \right.\right\}, $$
 \noindent where 
 $$\mathbf{s}=(\underbrace{\mu_1,\ldots,\mu_1}_{r_1\text{ times}},\underbrace{\mu_2,\ldots,\mu_2}_{r_2 \text{ times}},\ldots,\underbrace{\mu_\ell,\ldots,\mu_\ell}_{r_\ell \text{ times}})\in \mathbb{Q}^r.$$
Similarly, for every permutation $w\in \mathfrak{S}_r$ we define
 $$\square_t^{w}=\left\{(\lambda_1,\ldots,\lambda_r)\in \widehat{\Delta}_{r-1}\subseteq \mathbb{R}^r\;\left|\; \sum_{i=2}^r \sigma_{w(i-1)}\lambda_i + \sigma_{w(r)}\lambda_1 \geq t \right.\right\}, $$
 \noindent where 
$$\boldsymbol{\sigma}=(\underbrace{\mu_\ell,\ldots,\mu_\ell}_{r_\ell\text{ times}},\underbrace{\mu_{\ell-1},\ldots,\mu_{\ell-1}}_{r_{\ell-1} \text{ times}},\ldots,\underbrace{\mu_1,\ldots,\mu_1}_{r_1 \text{ times}})\in \mathbb{Q}^r.$$
By abuse of notation, we will also denote by $\square_t$ and $\square_t^{w}$ the full dimensional polytopes in $\mathbb{R}^{r-1}$ obtained via the projection of $\widehat{\Delta}_{r-1}\subseteq \mathbb{R}^r$ onto $\Delta_{r-1}\subseteq \mathbb{R}^{r-1}$. Explicitly, if $(\nu_2,\ldots,\nu_r)$ are coordinates in $\mathbb{R}^{r-1}$ then
 $$\square_t=\left\{(\nu_2,\ldots,\nu_r)\in \Delta_{r-1}\subseteq \mathbb{R}^{r-1}\;\left|\;s_1\left(1-\sum_{i=2}^r \nu_i\right)+\sum_{i=2}^r s_i\nu_i \geq t  \right.\right\} $$
 \noindent and
 \begin{equation*} \square_t^{w}=\left\{(\nu_2,\ldots,\nu_r)\in \Delta_{r-1}\;\left|\; \sum_{i=2}^r \sigma_{w(i-1)}\nu_i+\sigma_{w(r)}\left(1-\sum_{i=2}^r \nu_i\right) \geq t  \right.\right\}. \tag{$\star$} \label{eq:star} \end{equation*}
\end{notation}

\begin{cor}\label{cor: volume Okounkov bodies}
 For any admissible flag $Y_\bullet$ on $\mathbb{P}(E)$ and any big rational class $\xi-tf$ on $\mathbb{P}(E)$ we have that 
 $$\vol_{\mathbb{R}^r}(\Delta_{Y_\bullet}({\xi-tf}))=\int_{\widehat{\square}_t} \left(\sum_{i=1}^\ell \mu_i \beta_i - t\right)\dfrac{\beta_1^{r_1-1}\cdots\beta_\ell^{r_\ell-1}}{(r_1-1)!\cdots (r_\ell-1)!}d\beta,$$
\noindent for $\widehat{\square}_t \subseteq \mathbb{R}^\ell$ as in Notation \ref{notation: Wolfe's square}.
\end{cor}

\begin{remark}\label{remark: augmented base locus}
Let $F=\pi^{-1}(q)\cong \mathbb{P}^{r-1}$ be any fiber of $\pi:\mathbb{P}(E)\to C$. Then, for any big $\mathbb{R}$-divisor $D$ on $\mathbb{P}(E)$ we have that $F\not\subseteq \mathbf{B}_{+}(D)$. In fact, if $D\sim_{\mathbb{R}}A+E$, with $A$ ample $\mathbb{R}$-divisor and $E$ effective $\mathbb{R}$-divisor, and if $F\subseteq \operatorname{Supp}(E)$, then we can write 
$E=aF+E'$ with $a>0$ and $F\not\subseteq \operatorname{Supp}(E')$, which implies that $F\not\subseteq \mathbf{B}_{+}(D)$ since $A+aF$ is ample.
\end{remark}

As a direct consequence, we compute the restricted volume function on a fiber $F=\pi^{-1}(q)\cong \mathbb{P}^{r-1}$.

\begin{cor}\label{cor: restricted volume}
 Let $F$ be a fiber of $\pi:\mathbb{P}(E)\to C$ and let $\xi-tf$ be a big rational class. Then,
 $$\vol_{\mathbb{P}(E)|F}(\xi-tf)=(r-1)!\cdot\int_{\widehat{\square}_t}\dfrac{\beta_1^{r_1-1}\cdots\beta_\ell^{r_\ell-1}}{(r_1-1)!\cdots (r_\ell-1)!}d\beta, $$
 
\noindent for $\widehat{\square}_t \subseteq \mathbb{R}^\ell$ as in Notation \ref{notation: Wolfe's square}. In particular, if $0\leq \tau \leq \mu_\ell-t$ then we have that 
$$\vol_{\mathbb{R}^{r-1}}(\Delta_{Y_\bullet}({\xi-tf})_{\nu_1=\tau})=\int_{\widehat{\square}_{t+\tau}}\dfrac{\beta_1^{r_1-1}\cdots\beta_\ell^{r_\ell-1}}{(r_1-1)!\cdots (r_\ell-1)!}d\beta, $$
\noindent where $Y_\bullet$ is any admissible flag on $\mathbb{P}(E)$ with divisorial component $Y_1=F$. In particular, these volumes depend only on $\operatorname{gr}(\operatorname{HN}_\bullet(E))$, the graded vector bundle associated to the Harder-Narasimhan filtration of $E$.
\end{cor}

\begin{proof}
 We consider 
 $$v(\tau)=\vol_{\mathbb{P}(E)}(\xi-tf+\tau f)=r!\cdot \int_{\widehat{\square}_{t-\tau}} \left(\sum_{i=1}^\ell \mu_i \beta_i - t+\tau\right)\dfrac{\beta_1^{r_1-1}\cdots\beta_\ell^{r_\ell-1}}{(r_1-1)!\cdots (r_\ell-1)!}d\beta.$$
 Since $F \not\subseteq \mathbf{B}_{+}(\xi-tf)$ by Remark \ref{remark: augmented base locus}, Theorem \ref{slices} and differentiation under the integral sign give
 $$\vol_{\mathbb{P}(E)|F}(\xi-tf)=\dfrac{1}{r}\cdot\left.\dfrac{d}{d\tau}v(\tau)\right|_{\tau=0}=(r-1)!\cdot \displaystyle\int_{\widehat{\square}_{t}} \dfrac{\beta_1^{r_1-1}\cdots\beta_\ell^{r_\ell-1}}{(r_1-1)!\cdots (r_\ell-1)!}d\beta.$$
\end{proof}

\begin{lemma}\label{lemma: Huayi volume slices}
 Following Notation \ref{notation: Wolfe's square} \eqref{eq:star}, we have that
 $$\vol_{\mathbb{R}^r}(\Delta_{Y_\bullet}(\xi-tf))=\int_{\square_t}\left(\sum_{j=1}^r s_j \lambda_j - t\right)d\lambda $$
 \noindent and
 $$\vol_{\mathbb{R}^{r-1}}(\Delta_{Y_\bullet}({\xi-tf})|_{\nu_1=\tau})=\int_{\square_{t+\tau}}d\lambda=\vol_{\mathbb{R}^{r-1}}(\square_{t+\tau}). $$
 Moreover, $\vol_{\mathbb{R}^{r-1}}(\square_{t+\tau})=\vol_{\mathbb{R}^{r-1}}(\square_{t+\tau}^{w})$ for every $w\in \mathfrak{S}_r$.
\end{lemma}

\begin{proof}
The first two equalities follow from Remark \ref{remark: Huayi volume}. For the last assertion consider $\square_{t}\subseteq \mathbb{R}^r$ and $\square_{t}^{w}\subseteq \mathbb{R}^r$ as in Notation \ref{notation: Wolfe's square}. Then, there is a linear transformation $T:\mathbb{R}^r \to \mathbb{R}^r$, whose associated matrix in the canonical basis of $\mathbb{R}^r$ is given by a permutation matrix, such that $T(\square_t)=\square_t^{w}$ and $|\det(T)|=1$.
\end{proof}

\section{Newton-Okounkov bodies on projective bundles over curves}

Let $E$ be a vector bundle on a smooth projective curve $C$, of rank $r\geq 2$ and degree $d$. In this section we study the  geometry of Newton-Okounkov bodies of rational big classes in $\N^1(\mathbb{P}(E))$ in terms of the numerical information of the Harder-Narasimhan filtration of $E$. In particular, Theorem A will be a consequence of Lemma \ref{lemma:reduction to big nef}, Theorem \ref{theo: okounkov bodies} and Corollary \ref{cor: main result}. We follow Notation \ref{notation: HN quotients}. 

\begin{defn}[Linear flag]\label{defi: natural flag} A complete flag of subvarieties $Y_\bullet$ on the projective vector bundle $\mathbb{P}(E)\xrightarrow{\pi} C$ is called a {\it linear flag centered at $p\in \mathbb{P}(E)$, over the point $q\in C$} (or simply a {\it linear flag}) if $Y_0=\mathbb{P}(E)$, $Y_1=\pi^{-1}(q)\cong \mathbb{P}^{r-1}$ and $Y_i \cong \mathbb{P}^{r-i}$ is a linear subspace of $Y_{i-1}$ for $i=2,\ldots,r$, with $Y_r=\{p\}$.
\end{defn}

\subsection{Ruled surfaces}

Let us begin with the following example that illustrates the general case. Namely, that the shape of Newton-Okounkov bodies on $\mathbb{P}(E)$ will depend on the semi-stability of $E$.

\begin{exmp}\label{ex:ruled surfaces} Suppose that $\rk(E)= 2$ and let $\eta=a(\xi-\mu_\ell f)+bf \in \N^1(\mathbb{P}(E))_\mathbb{Q}$ be a big class. In other words, $a,b\in \mathbb{Q}_{>0}$. Let 
$$Y_\bullet: \mathbb{P}(E) \supseteq F=\pi^{-1}(q) \supseteq \{p\}$$
\noindent be the linear flag centered at $p\in \mathbb{P}(E)$, over $q\in C$. The Newton-Okounkov body of $\eta$ can be computed by applying \cite[Theo. 6.4]{LM09} (see Example \ref{ex:surfaces}).

\medskip

\noindent\textit{Semi-stable case.} If $E$ is semi-stable then Corollary \ref{Miyaoka} implies that every big class is ample. In particular, for every $\eta \in \Bg(X)_{\mathbb{Q}}$ and every $\mathbb{Q}$-divisor $D_\eta$ with numerical class $\eta$, we have that $N(D_\eta)=0$ and $P(D_\eta)=D_\eta$. It follows that, with the notation as in Example \ref{ex:surfaces}, $\nu=0$, $\tau_F(\eta)=b$, $\alpha(t)=0$ for every $t\in [0,b]$ and $\beta(t)=a$ for every $t\in [0,b]$. The Newton-Okounkov bodies are given by rectangles in this case.

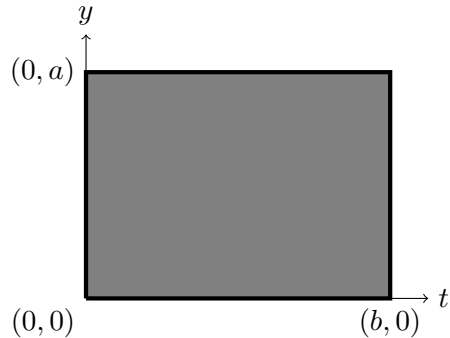
\begin{figure}[H]
\begin{center}
\begin{tikzpicture}
%
%
\draw [->]  (0,0) -- (4.5,0);   
\node [right] at (4.5,0) {$t$};
\draw [->]  (0,0) -- (0,3.5); 
\node [above] at (0,3.5) {$y$}; 
\draw [fill=gray, ultra thick] (0,0) -- (0,3) -- (4,3) -- (4,0) -- (0,0);  

%
%
\node [below left] at (0,0) {$(0,0)$};
\node [left] at (0,3) {$(0,a)$};
\node [below] at (4,0) {$(b,0)$};
\end{tikzpicture} 
\end{center}
\caption{Newton-Okounkov body $\Delta_{Y_\bullet}(\eta)$ for $E$ semi-stable}\label{fig:1}
\end{figure}

\medskip

\noindent \textit{Unstable case.} If $E$ is unstable we consider its Harder-Narasimhan filtration
$$0 \to E_1 \to E \to Q_1 \to 0, $$
and we note that in this case $E_1 = Q_\ell$. The quotient $E\to Q_1 \to 0$ corresponds to a section $s:C\to \mathbb{P}(E)$ with $[s(C)]=\xi-\mu_\ell f$ in $\N^1(\mathbb{P}(E))$. The curve $s(C)$ is the only irreducible curve on $\mathbb{P}(E)$ with negative self-intersection. On the other hand, if $D_{\eta}$ is any $\mathbb{Q}$-divisor with numerical class $\eta$ then, either $\eta$ is inside the nef cone of $\mathbb{P}(E)$ and thus $P(D_\eta)=D_\eta$ and $N(D_\eta)=0$, or $\eta$ is big and not nef in which case we compute that 
$$[P(D_\eta)]=\frac{b}{\mu_\ell - \mu_1}(\xi-\mu_1 f) \;\; \text{and}\;\; [N(D_\eta)]=\left(\frac{a(\mu_\ell - \mu_1)-b}{\mu_\ell - \mu_1} \right)(\xi - \mu_\ell f) $$
in $\N^1(\mathbb{P}(E))_\mathbb{Q}$. We notice that $N(D_\eta)=\left(\frac{a(\mu_\ell - \mu_1)-b}{\mu_\ell - \mu_1} \right)\cdot s(C)$ as $\mathbb{Q}$-divisor, by minimality of the negative part and by the fact that the negative part is unique in its numerical equivalence class. See \cite[Lemm. 14.10, Cor. 14.13]{Bad01} for details.

Let $t^*=b-a(\mu_\ell-\mu_1)$. We have that $\eta$ is big and nef if and only if $t^*\geq 0$, in which case the class $\eta - tf$ is big and nef for $0\leq t \leq t^*$. For $t^* \leq t \leq b$, the same computation above enables us to find the Zariski decomposition of $D_\eta - t F$, which is big and not nef.

We notice that, with the notation as in Example \ref{ex:surfaces}, $\nu=0$. On the other hand, the functions $\alpha(t)$ and $\beta(t)$ will depend on whether or not we have $\{p\}=s(C)\cap F$. A straightforward computation shows that the Newton-Okounkov bodies are given by the following finite polygons in $\mathbb{R}^2$.

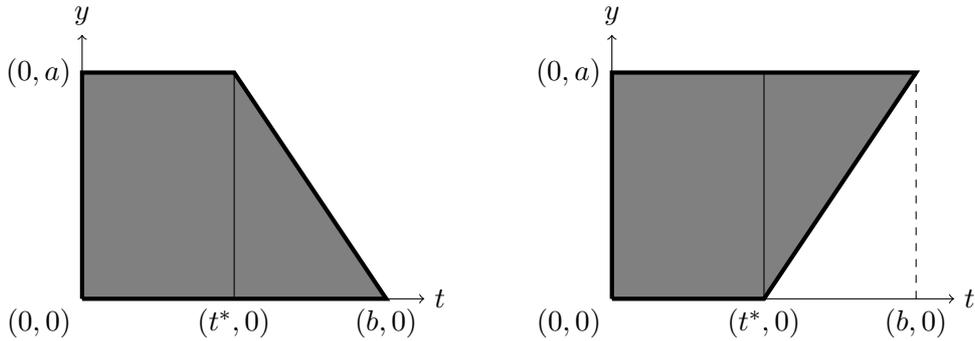
\begin{figure}[H]
\begin{center}
\begin{tikzpicture}
%
%
\draw [->]  (0,0) -- (4.5,0);   
\node [right] at (4.5,0) {$t$};
\draw [->]  (0,0) -- (0,3.5); 
\node [above] at (0,3.5) {$y$}; 
\draw [fill=gray, ultra thick] (0,0) -- (0,3) -- (2,3) -- (4,0) -- (0,0);  
\draw (2,0) -- (2,3); 

%
%
\node [below left] at (0,0) {$(0,0)$};
\node [below] at (2,0) {$(t^*,0)$};
\node [left] at (0,3) {$(0,a)$};
\node [below] at (4,0) {$(b,0)$};
\end{tikzpicture} 
\hspace{7mm}
\begin{tikzpicture}
%
%
\draw [->]  (0,0) -- (4.5,0);   
\node [right] at (4.5,0) {$t$};
\draw [->]  (0,0) -- (0,3.5); 
\node [above] at (0,3.5) {$y$}; 
\draw [fill=gray, ultra thick] (0,0) -- (0,3) -- (4,3) -- (2,0) -- (0,0);  
\draw (2,0) -- (2,3); 
\draw [dashed] (4,0) -- (4,3);
%
%
\node [below left] at (0,0) {$(0,0)$};
\node [below] at (2,0) {$(t^*,0)$};
\node [left] at (0,3) {$(0,a)$};
\node [below] at (4,0) {$(b,0)$};
\end{tikzpicture} 
\end{center}
\caption{$\Delta_{Y_\bullet}(\eta)$ for $E$ unstable and $\eta$ big and nef \\ {(a) if $\{p\}\neq s(C)\cap F $} \hspace{3.5cm} (b) if $\{p\}=s(C)\cap F$}\label{fig:2}
\end{figure}

\begin{figure}[H]
\begin{center}
\begin{tikzpicture}
%
%
\draw [->]  (0,0) -- (2.5,0);   
\node [right] at (2.5,0) {$t$};
\draw [->]  (0,0) -- (0,3.5); 
\node [above] at (0,3.5) {$y$}; 
\draw [fill=gray, ultra thick] (0,0) -- (0,3) -- (2,0) -- (0,0);

%
%
\node [below left] at (0,0) {$(0,0)$};
\node [left] at (0,3) {$(0,\frac{b}{\mu_\ell-\mu_1})$};
\node [below] at (2,0) {$(b,0)$};
\end{tikzpicture} 
\hspace{1cm}
\begin{tikzpicture}
%
%
\draw [->]  (0,0) -- (2.5,0);   
\node [right] at (2.5,0) {$t$};
\draw [->]  (0,0) -- (0,3.5); 
\node [above] at (0,3.5) {$y$}; 
\draw [fill=gray, ultra thick] (0,0.5) -- (0,3) -- (2,3) -- (0,0.5);  
\draw [dashed] (2,0) -- (2,3); 
%
%
\node [below left] at (0,0) {$(0,0)$};
\node [left] at (0,0.5) {$(0,\frac{-t^*}{\mu_\ell-\mu_1})$};
\node [left] at (0,3) {$(0,a)$};
\node [below] at (2,0) {$(b,0)$};
\end{tikzpicture} 
\end{center}
\caption{$\Delta_{Y_\bullet}(\eta)$ for $E$ unstable and $\eta$ big and not nef \\ {(a) if $\{p\}\neq s(C)\cap F $} \hspace{3.5cm} (b) if $\{p\}=s(C)\cap F$}\label{fig:3}
\end{figure}
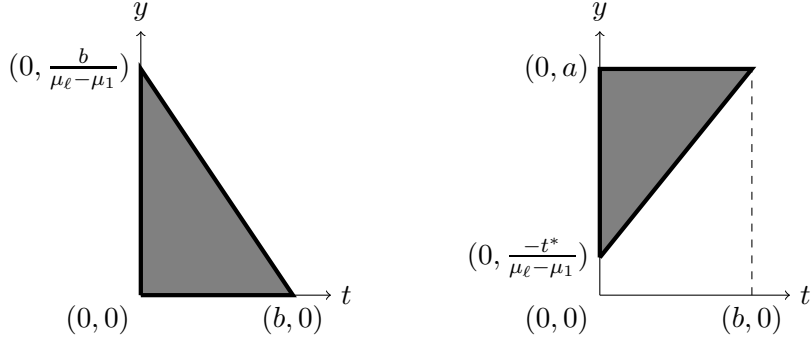 

We notice that Figure \ref{fig:3} provides examples of big and not nef divisors classes $\eta$ such that the origin $\vec{0}\in \Delta_{Y_\bullet}(\eta)$ for almost every linear flag $Y_\bullet$ except for one. This shows in particular that condition (2) in the characterization of nefness given in \cite[Cor. 2.2]{KL15b} has to be checked for {\it all} linear flags.

\end{exmp}

\subsection{Some reductions}

We first observe that Proposition B in $\S 1$ states that the shape of Newton-Okounkov bodies on projective semi-stable vector bundles will be similar as in Example \ref{ex:ruled surfaces} above. Moreover, this provides a characterization of semi-stability in terms of Newton-Okounkov bodies. 

\begin{proof}[Proof of Proposition B]
 
\noindent $(1)\Rightarrow (2)$. Let $\eta=a(\xi-\mu_\ell f)+bf$ be a big rational class on $\mathbb{P}(E)$ and $Y_\bullet: X = Y_0 \supseteq Y_1 \supseteq Y_2 \supseteq \cdots \supseteq Y_{r-1} \supseteq Y_r = \{p\}$ be a linear flag centered at $p\in \mathbb{P}(E)$, over the point $q\in C$. Since $E$ is semi-stable we have that $\Bg(\mathbb{P}(E))=\Amp(\mathbb{P}(E))$ by Corollary \ref{Miyaoka}. Equivalently, we have that $\mathbf{B}_+(\eta)= \emptyset$ for every big rational class $\eta$. 

We notice that $\tau_F(\eta)=\sup\{s>0\;|\;\eta - s f \in \Bg(\mathbb{P}(E)) \}=b$. The implication follows from Proposition \ref{propo:psef=nef} and \cite[Cor. 4.11]{Bo12}, by noting that if $D_\eta$ is a $\mathbb{Q}$-divisor with numerical class $\eta$ then 
$$\Delta_{Y_\bullet}(\eta)_{\nu_1=t}=\Delta_{Y_\bullet|F}(F,(D_\eta-tF)|_F)=a \Delta_{Y_\bullet|F}(F,H) = a\Delta_{r-1} \subseteq \mathbb{R}^{r-1},$$
where $H\subseteq F\cong \mathbb{P}^{r-1}$ is an hyperplane section. 
\medskip

\noindent $(2)\Rightarrow (1)$. We notice that if for all linear flags $Y_\bullet$ the Newton-Okounkov body of $\eta=a(\xi-\mu_\ell f)+bf$ is given by $\Delta_{Y_\bullet}(\eta)=[0,b]\times a\Delta_{r-1} \subseteq \mathbb{R}^r$ then $\eta$ is a big and nef class, by \cite[Cor. 2.2]{KL15b}. We can therefore compute the volume of $\eta$ via the top self-intersection $\vol_{\mathbb{P}(E)}(\eta)=\eta^r = r a^{r-1}\left(b-a(\mu_\ell-\mu(E))\right)$. On the other hand, we have that
\begin{equation*}
 \begin{aligned}
\vol_{\mathbb{R}^n}\left(\Delta_{Y_\bullet}(\eta) \right) = \vol_{\mathbb{R}^r}\left([0,b]\times a\Delta_{r-1} \right) = \frac{a^{r-1}}{(r-1)!}b,
 \end{aligned}
\end{equation*}
 and hence \cite[Theo. A]{LM09} leads to $\mu_\ell = \mu(E)$, implying the result.
\end{proof}

We will first reduce our problem to computing the Newton-Okounkov body of the big and nef divisor class $\xi-\mu_1 f$.

\begin{lemma}\label{lemma:reduction to big nef}
 Let $E$ be a unstable vector bundle on a smooth projective curve $C$, with $\rk(E)=r\geq 2$. Then for every big class $\eta = a(\xi-\mu_\ell f)+bf$ and every linear flag $Y_\bullet$ we have that
 \begin{itemize}
  \item[(1)] $\Delta_{Y_\bullet}(\eta)=\left([0,t^*]\times a\Delta_{r-1} \right) \cup \left(a \Delta_{Y_\bullet}(\xi-\mu_1 f)+t^*\vec{e}_1\right)$ if $\eta$ is big and nef;
  \item[(2)] $\Delta_{Y_\bullet}(\eta)=a \Delta_{Y_\bullet}(\xi-\mu_1 f)_{\nu_1 \geq - t^*}+t^*\vec{e}_1$ if $\eta$ is big and not nef.
 \end{itemize}
 Here $t^*=b-a(\mu_\ell-\mu_1)$ and $a\Delta_{r-1}\subseteq \mathbb{R}^{r-1}$ is the $(r-1)$-simplex with side length $a$. 
\end{lemma}

\begin{proof}
 If $\eta$ is an ample rational class then ${\Delta_{Y_\bullet}(\eta)_{\nu_1=t}=a\Delta_{r-1}}$ for $0\leq t \leq t^*$, by Proposition \ref{propo:psef=nef} and \cite[Exam. 1.1]{LM09}. On the other hand, we have that ${\Delta_{Y_\bullet}(\eta)_{\nu_1 \geq t^*}= a \Delta_{Y_\bullet}(\xi - \mu_1 f)+t^*\vec{e}_1}$ by Theorem \ref{slices}.
 
 If $\eta$ is big and not nef then $t^*<0$. Theorem \ref{slices} implies therefore that ${a \Delta_{Y_\bullet}(\xi - \mu_1 f)_{\nu_1\geq -t^*}= \Delta_{Y_\bullet}(\eta)-t^*\vec{e}_1}$, which leads to (2).
\end{proof}

\subsection{A toric computation}

We will need the following result concerning the Newton-Okounkov bodies of some toric graded algebras.

\begin{lemma}\label{lemm: NO of special algebras}
Fix $A\in \mathbb{R}$, an integer $r\geq 2$, $\boldsymbol{\sigma}=(\sigma_1,\ldots,\sigma_r)\in \mathbb{R}^r$ and homogeneous coordinates $[x_1:\ldots:x_r]$ on $\mathbb{P}^{r-1}$. Put $B_0=k$ and consider for every integer $m\geq 1$ the vector subspace $B_m\subseteq \h^0(\mathbb{P}^{r-1},\mathcal{O}_{\mathbb{P}^{r-1}}(m))$ generated by monomials ${x^\alpha=x_1^{\alpha_1}\cdots x_r^{\alpha_r}}$ of total degree $|\alpha|=m$ such that $\sum_{i=1}^r \sigma_i \alpha_i > A$. Suppose that $A\geq 0$ and $\sigma_1\geq \sigma_2 \geq \cdots \geq \sigma_r$. Then
 \begin{enumerate}
  \item $B_\bullet=\oplus_{m\geq 0}B_m$ is a graded subalgebra of the coordinate ring $k[x_1,\ldots,x_r]$.
  \item Let us denote by $V_\bullet$ the flag of linear subspaces 
  $$V_\bullet: V_0=\mathbb{P}^{r-1}\supseteq V_1 \supseteq \cdots \supseteq V_{r-1},$$
  where $V_i=\{x_1=\ldots=x_i=0\}$ for $i=1,\ldots,r-1$ and consider the Schubert cell decomposition of the full flag variety parametrizing complete linear flag on $\mathbb{P}^{r-1}$
  $$\mathbb{F}_r = \coprod_{w\in \mathfrak{S}_r} \Omega_w, $$
  with respect to $V_\bullet$. Fix a permutation $w\in \mathfrak{S}_r$ and let $Y_\bullet \in \Omega_w$. Then $\Delta_{Y_\bullet}(B_\bullet)$ is given by the projection of the polytope
  $$\left\{(\nu_1,\ldots,\nu_r)\in \mathbb{R}^r_{\geq 0}\;\left|\;\sum_{i=1}^r \nu_i=1 \text{ and } \sum_{i=1}^r \sigma_{w(i)} \nu_i \geq 0 \right.\right\} $$
  onto the hyperplane $\{\nu_r=0\}$.
 \end{enumerate}
\end{lemma}
\begin{proof}
Let us suppose that $x^\alpha \in B_m$ and $x^{\alpha'}\in B_{m'}$, then $x^{\alpha+\alpha'}$ belongs to $b_{m+m'}$ since we have $\sum_{i=1}^r \sigma_i (\alpha_i+\alpha_i') > 2A \geq A$ as long as $A\geq 0$. This proves (1).

In order to prove (2) we will first show that every automorphism $\varphi$ of $\mathbb{P}^{r-1}$ that fixes the flag $V_\bullet$ induces an automorphism of the graded algebra $B_\bullet$. For this, we remark that automorphisms $\varphi$ of $\mathbb{P}^{r-1}$ fixing the flag $V_\bullet$ correspond to lower triangular matrices in $\operatorname{PGL}_r(k)$ and hence, given $\varphi=(\varphi_{i,j})_{1\leq i,j \leq r}$ such a matrix, the image of the monomial $x^\alpha \in B_m$ via the induced action on $k[x_1,\ldots,x_r]$ is given by
$$\varphi(x^\alpha)=\prod_{i=1}^r (\varphi_{i,1}x_1+\ldots+\varphi_{i,i}x_i)^{\alpha_i}.$$
The above product can be written as a linear combination of monomials of the form $x^{\alpha'}=x^{k_1+\ldots+k_r}$ where $k_i=(k_{i,1},\ldots,k_{i,i},0,\ldots,0)\in \mathbb{N}^r$ is such that $|k_i|=\alpha_i$. Let us prove that all these monomials belong to $B_m$ as well. In fact, we have that $\alpha_i'=\sum_{j=i}^r k_{j,i}$ and hence
$$ 
\begin{array}{ccll}
\sum_{i=1}^r \sigma_i \alpha_i' & = & \sum_{i=1}^r \sum_{j=i}^r  \sigma_i k_{j,i}=\sum_{i=1}^r \sum_{j=1}^i \sigma_j k_{i,j} \\
 & \geq & \sum_{i=1}^r \sum_{j=1}^i \sigma_i k_{i,j} & \text{ since } \sigma_1\geq\cdots \geq \sigma_r \\
 & = & \sum_{i=1}^r \sigma_i |k_i| = \sum_{i=1}^r \sigma_i \alpha_i > A.
\end{array}
$$
It follows that $\varphi$ induces an automorphism of the graded algebra $B_\bullet$.

In order to compute Newton-Okounkov bodies with respect to linear flags on a given Schubert cell, we note that \cite[Prop. 1.2.1]{Bri05} implies that given a permutation $w\in \mathfrak{S}_r$ and a linear flag $Y_\bullet \in \Omega_w$ there exists an automorphism $\varphi$ of $\mathbb{P}^{r-1}$ that fixes the reference flag $V_\bullet$ and such that the image of the flag $Y_\bullet^w$ via the induced action of $\varphi$ on $\mathbb{F}_r$ is $Y_\bullet$, where $Y_\bullet^w$ is the linear flag such that for every $i=1,\ldots,r-1$ we have
$$Y_i^w = \{x_{w(1)}=\ldots=x_{w(i)}=0 \} \subseteq \mathbb{P}^{r-1}.$$
It follows from the previous paragraph that $\varphi$ induces an automorphism of the graded algebra $B_\bullet$ and then for every $m\geq 1$ and every $P\in B_m$ we have that
$$\nu_{Y_\bullet}(\varphi(P))=\nu_{\varphi(Y_\bullet^w)}(\varphi(P))=\nu_{Y_\bullet^w}(P).$$
In particular, we have $\{ \nu_{Y_\bullet}(P) \}_{P\in B_{m}} = \{ \nu_{Y_\bullet^w}(P) \}_{P\in B_{m}}\subseteq \mathbb{N}^{r-1}$ and consequently ${\Delta_{Y_\bullet}(B_{\bullet})=\Delta_{Y_\bullet^w}(B_{\bullet})}$. Therefore, we can suppose that $Y_\bullet=Y_\bullet^w\in \Omega_w$ in order to prove (2).

Since $B_m$ is generated by monomials we have that
$$\Delta_{Y_\bullet^w}(B_{\bullet}) = \overline{\bigcup_{m\geq 1}\left\{\left.\frac{\nu_{Y_\bullet^w}(x^\alpha)}{m}\;\right|\;x^\alpha\in B_{m} \right\}}.$$
We compute for every monomial $x^\alpha=x_1^{\alpha_1}\cdots x_r^{\alpha_r}\in B_m$ that
$$\nu_{Y_\bullet^w}(x^\alpha)=(\alpha_{w(1)},\alpha_{w(2)},\ldots,\alpha_{w(r-1)})\in \mathbb{N}^{r-1} $$
and hence for every $m\geq 1$ the set $\left\{\left.\frac{\nu_{Y_\bullet^w}(x^\alpha)}{m}\;\right|\;x^\alpha\in B_{m} \right\}$ is given by the set of points of the form $\left(\frac{\alpha_{w(1)}}{m},\ldots,\frac{\alpha_{w(r-1)}}{m} \right) \in \mathbb{Q}^{r-1}_{\geq 0}$ where $\alpha=(\alpha_1,\ldots,\alpha_r)\in \mathbb{N}^r$ is such that $|\alpha|=\alpha_1+\ldots+\alpha_r=m$ and $\sum_{i=1}^r \sigma_i \alpha_i > A$. Equivalently, is given by the set of points $\left(\frac{\alpha_{w(1)}}{m},\ldots,\frac{\alpha_{w(r-1)}}{m} \right) \in \mathbb{Q}^{r-1}_{\geq 0}$ such that
\begin{itemize} 
\item[$\circ$] $\frac{\alpha_{w(1)}}{m}+\ldots+\frac{\alpha_{w(r-1)}}{m}\leq 1$ and 
\item[$\circ$] ${\textstyle \begin{array}{ll} \sum_{i=1}^r \sigma_{i} \frac{\alpha_{i}}{m} &=\sum_{i=1}^r \sigma_{w(i)} \frac{\alpha_{w(i)}}{m}\\
& = \sum_{i=1}^{r-1} \sigma_{w(i)} \frac{\alpha_{w(i)}}{m} + \sigma_{w(r)}\left(1-\sum_{i=1}^{r-1}  \frac{\alpha_{w(i)}}{m}  \right) > \frac{A}{m}. \end{array}}$
\end{itemize}
We conclude therefore that
$$\Delta_{Y_\bullet^w}(B_{\bullet}) = \left\{(\nu_1,\ldots,\nu_{r-1})\in \Delta_{r-1}\;\left|\; \sum_{i=1}^{r-1} \sigma_{w(i)} \nu_i + \sigma_{w(r)}\left(1-\sum_{i=1}^{r-1} \nu_i \right) \geq 0 \right.\right\},$$
from which (2) follows.
\end{proof}

\begin{remark}
The assumption $\sigma_1\geq \sigma_2 \geq \cdots \geq \sigma_r$ in Lemma \ref{lemm: NO of special algebras} can always be fulfilled by considering an automorphism of $\mathbb{P}^{r-1}$ permuting the chosen homogeneous coordinates.
\end{remark}

\subsection{Proofs of main results}

Let us note that given $q\in C$, the Harder-Narasimhan filtration
$$\operatorname{HN}_\bullet(E) : 0=E_\ell \subseteq E_{\ell-1} \subseteq \cdots \subseteq E_1 \subseteq E_0=E$$
induces a (not necessarily complete) flag on $\mathbb{P}(E)$ in the following way: let $Y_0=\mathbb{P}(E)$ and for every $i=2,\ldots,\ell$ the exact sequence
$$0\to Q_i\to E/E_i\to E/E_{i-1}\to 0 $$
\noindent induces an inclusion $\mathbb{P}((E/E_{i-1})|_q)\hookrightarrow \mathbb{P}((E/E_i)|_q) $. We obtain therefore the following (possibly partial) flag of linear subvarieties
$$\mathbb{P}((E/E_1)|_q)\subseteq \mathbb{P}((E/E_2)|_q)\subseteq \cdots \subseteq \mathbb{P}((E/E_{\ell-1})|_q)\subseteq \mathbb{P}(E|_q)=\pi^{-1}(q)\subseteq \mathbb{P}(E) $$
\noindent with $\operatorname{codim}_{\mathbb{P}(E)}\mathbb{P}((E/E_i)|_q)=\operatorname{rank}(E_i)+1$. We also note that this flag is complete if and only if all the semi-stable quotients $Q_i$ are line bundles. In general, it will be necessary to choose a complete linear flag on each $\mathbb{P}((E_{i-1}/E_i)|_q)=\mathbb{P}(Q_i|_q)$ in order to complete the flag above.

We shall consider linear flags that are compatible with the Harder-Narasimhan filtration of $E$ in the sense that they complete the previous flag.

\begin{defn}[Compatible linear flag]\label{defi: compatible natural flags} A linear flag $Y_\bullet$ on $\mathbb{P}(E)$ over $q\in C$ is said to be {\it compatible with the Harder-Narasimhan filtration of $E$} if 
$$Y_{\operatorname{rank}E_i+1}=\mathbb{P}((E/E_i)|_q)\cong \mathbb{P}^{r-\operatorname{rank}E_i-1} \subseteq \mathbb{P}(E)$$
\noindent for every $i=1,\ldots,\ell$. 
\end{defn}

We will adopt the following convention.

\begin{convention}\label{convention:schubert cells}
 Let $V_\bullet$ be a fixed linear flag on $\mathbb{P}(E)$ over $q\in C$ and consider the corresponding Schubert cell decomposition 
 $$\mathbb{F}_r = \coprod_{w\in \mathfrak{S}_r} \Omega_w, $$
 where $\mathbb{F}_r$ is the full flag variety parametrizing complete flags of linear subspaces of $V_1 = \pi^{-1}(q)\cong \mathbb{P}^{r-1}$. We say that a linear flag $Y_\bullet$ on $\mathbb{P}(E)$ over $q\in C$ {\it belongs} to a Schubert cell $\Omega_w$ if the induced linear flag $Y_\bullet|Y_1$ belongs to $\Omega_w$.
\end{convention}

We can prove now the main reduction step. Namely, compute the Newton-Okounkov body of the big and nef class $\xi-\mu_1f$ with respect to any linear flag.

\begin{thm}\label{theo: okounkov bodies}
Let $C$ be a smooth projective curve and let $E$ be a vector bundle over $C$ of rank $r\geq 2$. Fix a linear flag $Y_\bullet^{\HN}$ on $\mathbb{P}(E)$ over $q\in C$ which is compatible with the Harder-Narasimhan filtration of $E$ and let
$$\mathbb{F}_r=\coprod_{w\in \mathfrak{S}_r}\Omega_w $$
be the corresponding Schubert cell decomposition of the full flag variety $\mathbb{F}_r$ parametrizing linear flags on $\pi^{-1}(q)\cong \mathbb{P}^{r-1}$. Then, for every linear flag $Y_\bullet$ on $\mathbb{P}(E)$ over $q\in C$ that belongs to $\Omega_w$ we have that
$$\Delta_{Y_\bullet}(\xi-\mu_1f)=\left\{(\nu_1,\ldots,\nu_r)\in \mathbb{R}_{\geq 0}^r \;|\;0\leq \nu_1\leq \mu_\ell-\mu_1,\;(\nu_2,\ldots,\nu_r)\in \square_{\mu_1+\nu_1}^{w} \right\}, $$
where $\square_{\mu_1+\nu_1}^{w}\subseteq \mathbb{R}^{r-1}$ is the full dimensional polytope defined in Notation \mbox{\ref{notation: Wolfe's square} \eqref{eq:star}}, with $(\mu_1,\ldots,\mu_\ell)$ and $(r_1,\ldots,r_\ell)$ given by the Harder-Narasimhan \mbox{filtration} of $E$ as in Notation \ref{notation: HN quotients}.
\end{thm}

\begin{proof}
We follow Notation \ref{notation: HN quotients}. We first note that if $E$ is semi-stable then $\mu_1=\mu_\ell=\mu(E)$ and hence the Theorem follows from Proposition B. Let us suppose from now on that $E$ is unstable.

\vspace{2mm}

For the reader's convenience, the proof is subdivided into several steps. We first observe that it is enough to compute rational slices of $\Delta_{Y_\bullet}(\xi-\mu_1f)$. In order to do so, we consider the restricted algebra $A_\bullet$ whose Newton-Okounkov body computes the desired slice. After performing a suitable Veronese embedding $A_{n\bullet}\subseteq A_\bullet$, we define a graded subalgebra $B_{n\bullet}\subseteq A_{n\bullet}$ which turns out to be a toric graded algebra as in Lemma \ref{lemm: NO of special algebras}. A comparison of volumes leads to the result.

\vspace{2mm}

\noindent {\it Step 1. Reduction to rational slices.} It follows from Theorem \ref{slices} and Lemma \ref{lemma: psef cone} that the projection of the Newton-Okounkov body of $\xi-\mu_1f$ onto the first coordinate is given by 
$$\operatorname{pr}_1(\Delta_{Y_\bullet}(\xi-\mu_1f))=[0,\mu_\ell-\mu_1].$$

By continuity of slices of Newton-Okounkov bodies (cf. \cite[Lemm. 1.7]{KL15b}), it suffices to consider a fixed $t\in ]0,\mu_\ell-\mu_1[\cap \mathbb{Q}$ and show that the slice of $\Delta_{Y_\bullet}(\xi-\mu_1f)$ at $t\in ]0,\mu_\ell-\mu_1[\cap \mathbb{Q}$ is given by
 $$\Delta_{Y_\bullet}(\xi-\mu_1f)|_{\nu_1=t}=\square_{\mu_1+t}^{w}\subseteq \mathbb{R}^{r-1} $$
 \noindent for linear flags $Y_\bullet$ on $\mathbb{P}(E)$ that belong to the Schubert cell $\Omega_w$ (see Convention \ref{convention:schubert cells}). 
 
Let us fix from now on a permutation $w\in \mathfrak{S}_r$ and a linear flag $Y_\bullet$ on $\mathbb{P}(E)$, over $q\in C$, and let us denote by $Y_\bullet|F$ the induced flag on $F$. Suppose that $Y_\bullet|F$ belongs to the Schubert cell $\Omega_w$ with respect to the reference flag $Y_\bullet^{\HN}|F$. 

\vspace{2mm}

\noindent {\it Step 2. Restricted algebra $A_\bullet$.} Consider a $\mathbb{Q}$-divisor $D$ such that $[D]=\xi-\mu_1f$ in $\N^1(\mathbb{P}(E))_\mathbb{Q}$. For every integer $m\geq 1$, let us define the subspace
$$\begin{array}{ll}
A_m=A_{m,t}&=\h^0(\mathbb{P}(E)|F,\mathcal{O}_{\mathbb{P}(E)}(\lfloor m(D-tF)\rfloor)) \\
 &=\text{Im}(\h^0(\mathbb{P}(E),\mathcal{O}_{\mathbb{P}(E)}(\lfloor m(D-tF)\rfloor)) \xrightarrow{\operatorname{rest}} \h^0(F,\mathcal{O}_F(m)))\\
 &\subseteq \h^0(F,\mathcal{O}_F(m)).
\end{array} $$
If follows from \cite[Prop. 4.1, Rem. 4.25]{LM09} that the restricted algebra $A_\bullet$ above computes the desired slice
$$\Delta_{Y_\bullet|F}(A_\bullet)=\Delta_{Y_\bullet}(\xi-\mu_1f)|_{\nu_1=t},$$
and that for every integer $n\geq 1$ we have
$$\Delta_{Y_\bullet}(nD)|_{\nu_1=nt}=\Delta_{Y_\bullet|F}(n(D-tF))=n\Delta_{Y_\bullet|F}(D-tF)=n\Delta_{Y_\bullet}(D)|_{\nu_1=t}.$$

\vspace{2mm}

\noindent {\it Step 3. Veronese embedding $A_{n\bullet}\subseteq A_\bullet$.} We will consider $A_{n\bullet}=\{a_{nm}\}_{m\geq 0}$ instead of $A_\bullet=\{a_m \}_{m\geq 0}$, for $n$ fixed and divisible enough such that $n\mu_1\in \mathbb{Z}$ and $nt\in \mathbb{Z}$. We also note that for $0<t<\mu_\ell-\mu_1$ we have that 
$$\mu_{\max}(S^m E\otimes \mathcal{O}_C(-m(\mu_1+t)\cdot q))=m(\mu_\ell-\mu_1-t)>0$$
and
$$\mu_{\min}(S^m E\otimes \mathcal{O}_C(-m(\mu_1+t)\cdot q))=-mt<0. $$
Therefore, by considering $n$ above large enough we may also assume that ${\mu_{\max}(S^{nm} E\otimes \mathcal{O}_C(-nm(\mu_1+t)\cdot q))>2g-1}$ for every $m\geq 1$.
 
Let $[x_1:\cdots:x_r]$ be homogeneous coordinates on $F\cong \mathbb{P}^{r-1}$. Since $Y_\bullet$ is a linear flag on $\mathbb{P}(E)$, there is an isomorphism of graded algebras
$$\phi:\bigoplus_{m\geq 0}H^0(F,\mathcal{O}_F(m))\to k[x_1,\ldots,x_r] $$
such that $A_m$ can be regarded as a subspace of $k[x_1,\ldots,x_r]_m$, the $k$-vector space of homogeneous polynomials of degree $m$ in the variables $x_1,\ldots,x_r$, for all $m\geq 0$. Via this identification, $A_\bullet$ can be seen as a graded subalgebra of $k[x_1,\ldots,x_r]$. Moreover, the projection formula implies that we can identify $A_m$ with
$$\text{Im}\Big(\h^0(C,S^mE\otimes \mathcal{O}_C(-m(\mu_1+t) \cdot q)) \xrightarrow{\operatorname{rest}} \h^0(C,(S^mE\otimes \mathcal{O}_C(-m(\mu_1+t) \cdot q))|_q)\Big). $$

\vspace{2mm}

\noindent {\it Step 4. Toric graded subalgebra $B_{n\bullet}\subseteq A_{n\bullet}$.} We shall define a graded subalgebra $B_{n\bullet} \subseteq A_{n\bullet}$ for which we can explicitly compute that
$$\Delta_{Y_\bullet|F}(B_{n\bullet}) = n\square_{\mu_1+t}^{w},$$
and we will prove that $ \Delta_{Y_\bullet|F}(B_{n\bullet}) = \Delta_{Y_\bullet|F}(A_{n\bullet})=n\Delta_{Y_\bullet|F}(A_{\bullet})$. 

\medskip

In order to construct $B_{n\bullet}$ let us note that Proposition \ref{propo: HN filtration of S^mE} implies that for every $m\geq 1$ there is a filtration
$$F_\bullet : 0=F_L\subseteq F_{L-1} \subseteq \cdots \subseteq F_1 \subseteq F_0=S^{nm} E \otimes \mathcal{O}_C(-nm(\mu_1+t)\cdot q) $$
whose successive quotients have the form
$$F_{j-1}/F_j \cong Q_{\mathbf{m}(j),\mu_1+t}=S^{m_1}Q_1 \otimes \cdots \otimes S^{m_\ell}Q_\ell\otimes \mathcal{O}_C(-nm(\mu_1+t)\cdot q) $$
\noindent for some partition $\mathbf{m}(j)\in \mathbb{N}^\ell$ of $nm$, and $\mu(Q_{\mathbf{m}(j),\mu_1+t})\leq \mu(Q_{\mathbf{m}(j+1),\mu_1+t})$ for every $j\in \{1,\ldots,L\}$.

Let us define $J=J(m)\in\{1,\ldots,L\}$ to be the largest index such that ${\mu(Q_{\mathbf{m}(J),\mu_1+t})\leq 2g-1}$. We have that for every $j\in\{J,\ldots,L-1\}$, the short exact sequence
$$0\to F_{j+1}\otimes \mathcal{O}_C(-q)\to F_{j+1}\to F_{j+1}|_q\to 0 $$
\noindent gives an exact sequence in cohomology
$$0\to \h^0(C,F_{j+1}\otimes \mathcal{O}_C(-q))\to \h^0(C,F_{j+1})\to \h^0(C, F_{j+1}|_q)\to 0, $$
\noindent since we have that $h^1(C,F_{j+1}\otimes \mathcal{O}_C(-q))=0$, by Lemma \ref{lemma: properties slope}. In particular, we get for every $j\in \{J,\ldots,L-1\}$ a surjection $\h^0(C,F_{j+1})\to \h^0(C,F_{j+1}|_q)$. Therefore, let us consider the subspaces
$$B_{nm}=\text{Im}(\h^0(C,F_{J+1}) \xrightarrow{\operatorname{rest}} \h^0(C,F_{J+1}|_q) )=\h^0(C,F_{J+1}|_q)\subseteq A_{nm}. $$

Let us choose homogeneous coordinates $[x_1:\ldots:x_r]$ on $F$ such that
$$Y_{i+1}^{\HN}=\{x_1=\ldots = x_{i}=0 \}\subseteq \mathbb{P}(E)$$
for $i=1,\ldots,r-1$. In particular, we have that 
$$Y_{\rk E_i+1}^{\HN}=\mathbb{P}((E/E_i)|_q)=\{x_1=\ldots=x_{\rk E_i}=0\}\subseteq \mathbb{P}(E)$$
for $i=1,\ldots,\ell$ and therefore the degree 1 part of the isomorphism $\phi$, 
$$\phi_1:\h^0(F,\mathcal{O}_F(1))\cong \h^0(C,(E \otimes \mathcal{O}_C(-(\mu_1+t)\cdot q))|_q) \to k[x_1,\ldots,x_r]_1, $$
is such that for every $i=0,\ldots, \ell-1$ the image of the subspace
$$\h^0(C,(E_i\otimes \mathcal{O}_C(-(\mu_1+t)\cdot q))|_q)\subseteq \h^0(C,(E \otimes \mathcal{O}_C(-(\mu_1+t)\cdot q))|_q) $$
via $\phi_1$ coincide with the subspace generated by the variables $x_1,\ldots,x_{\rk E_i}$. By taking symmetric powers it follows from Proposition \ref{propo: HN filtration of S^mE} that for each $m\geq 1$ we have that ${B_{nm}\subseteq k[x_1,\ldots,x_r]_{nm}}$, the image of ${\h^0(C,F_{J+1}|_q)\subseteq \h^0(C,F_0|_q)}$, corresponds to the subspace of homogeneous polynomials of degree $nm$ generated by polynomials of the form
 $$P(\mathbf{x})=P_1(\mathbf{x}_1)\cdots P_\ell(\mathbf{x}_\ell) $$
\noindent where $P_i$ is an homogeneous polynomial of degree $m_i\geq 0$ in the variables $\mathbf{x}_i=(x_{i,1},\ldots,x_{i,{r_i}})$, where $(x_1,\ldots,x_r)=(\mathbf{x}_\ell,\ldots,\mathbf{x}_1)$, and the $m_i$ are such that $m_1+\ldots+m_\ell=nm$ and
 $$\mu_1m_1+\ldots+\mu_\ell m_\ell > nm(\mu_1+t)+2g-1. $$
In other words, $B_{nm}$ is the subspace generated by monomials $x^\alpha=x_1^{\alpha_1}\cdots x_r^{\alpha_r}$ of total degree $|\alpha|=nm$ such that $\sum_{i=1}^r (\sigma_i-\mu_1-t) \alpha_i > 2g-1$, where $\boldsymbol{\sigma}=(\underbrace{\mu_\ell,\ldots,\mu_\ell}_{r_\ell\text{ times}},\underbrace{\mu_{\ell-1},\ldots,\mu_{\ell-1}}_{r_{\ell-1} \text{ times}},\ldots,\underbrace{\mu_1,\ldots,\mu_1}_{r_1 \text{ times}})\in \mathbb{Q}^r$.

\vspace{2mm}
 
\noindent {\it Step 5. Volume comparison and conclusion.} It follows from Lemma \ref{lemm: NO of special algebras} applied\footnote{We note that if $C\cong \mathbb{P}^1$ then all the slopes $\mu_i=\mu(Q_i)$ are integer numbers and hence the inequality ``$>2g-1$'' becomes ``$\geq 0$''.} to the collection of subspaces $\{B_{nm} \}_{m\geq 1}$ that $B_{n\bullet}$ is a graded subalgebra of $k[x_1,\ldots,x_r]$ whose Newton-Okounkov body, with respect to a linear flag $Y_\bullet$ that belongs to the Schubert cell $\Omega_w$ (see Convention \ref{convention:schubert cells}) is given by
$$ \Delta_{Y_\bullet|F}(B_{n\bullet})=n\square_{\mu_1+t}^{w}, $$
\noindent where 
$$\square_{\mu_1+t}^{w}={\textstyle \Big\{(\nu_2,\ldots,\nu_r)\in \Delta_{r-1}\;\left|\; \sum_{i=2}^r \sigma_{w(i-1)}\nu_i+\sigma_{w(r)}\left(1-\sum_{i=2}^r \nu_i\right) \geq \mu_1+t  \right.\Big\}}$$
and $\boldsymbol{\sigma}=(\underbrace{\mu_\ell,\ldots,\mu_\ell}_{r_\ell\text{ times}},\underbrace{\mu_{\ell-1},\ldots,\mu_{\ell-1}}_{r_{\ell-1} \text{ times}},\ldots,\underbrace{\mu_1,\ldots,\mu_1}_{r_1 \text{ times}})\in \mathbb{Q}^r$. Finally, we have that
$$
\begin{array}{ll}
 \vol_{\mathbb{R}^{r-1}}(\Delta_{Y_\bullet|F}(B_{n\bullet}))&=\vol_{\mathbb{R}^{r-1}}(n\square_{\mu_1+t}^{w}) \\
 &=\vol_{\mathbb{R}^{r-1}}(n\Delta_{Y_\bullet}(\xi-\mu_1f)|_{\nu_1=t}) \text{ by Lemma \ref{lemma: Huayi volume slices}}\\
 &= \vol_{\mathbb{R}^{r-1}}(\Delta_{Y_\bullet|F}(A_{n\bullet}))
\end{array} $$
and hence the inclusion $\Delta_{Y_\bullet|F}(B_{n\bullet})\subseteq \Delta_{Y_\bullet|F}(A_{n\bullet})$ leads to the equality $\Delta_{Y_\bullet|F}(B_{n\bullet}) = \Delta_{Y_\bullet|F}(A_{n\bullet})$, as the two convex bodies have equal volume. From this we conclude that
 $$\Delta_{Y_\bullet}(\xi-\mu_1f)|_{\nu_1=t}=\square_{\mu_1+t}^{w}\subseteq \mathbb{R}^{r-1}. $$ \end{proof}

The following result (from which Theorem A is easily deduced) is an immediate consequence. We keep the same notation as in Theorem \ref{theo: okounkov bodies}.

\begin{cor}\label{cor: main result}
For every linear flag $Y_\bullet$ on $\mathbb{P}(E)$ that belongs to the Schubert cell $\Omega_w$ and every big rational class $\eta=a(\xi-\mu_\ell f)+bf$ we have that 
$$\Delta_{Y_\bullet}(\eta)=\left\{(\nu_1,\ldots,\nu_r)\in \mathbb{R}^r_{\geq 0}\;|\;0 \leq \nu_1 \leq b,\;(\nu_2,\ldots,\nu_r)\in a\square_{\mu_\ell - \frac{1}{a}(b-\nu_1)}^{w} \right\}, $$
and hence the global Newton-Okounkov body of $\mathbb{P}(E)$ with respect to $Y_\bullet$ is given by
$$\begin{array}{ll}
\Delta_{Y_\bullet}(\mathbb{P}(E))=&\Big\{((a(\xi-\mu_\ell f)+bf),(\nu_1,\ldots,\nu_r))\in \N^1(\mathbb{P}(E))_\mathbb{R}\times \mathbb{R}^{r} \text{ such that } \\
& \;\; 0 \leq \nu_1 \leq b \text{ and } (\nu_2,\ldots,\nu_r)\in a\square_{\mu_\ell - \frac{1}{a}(b-\nu_1)}^{w} \Big\}.   
\end{array} $$
In particular, the global Newton-Okounkov body $\Delta_{Y_\bullet}(\mathbb{P}(E))$ is a rational polyhedral cone and it depends only on $\operatorname{gr}(\operatorname{HN}_\bullet(E))$, the graded vector bundle associated to the Harder-Narasimhan filtration of $E$.
\end{cor}

\begin{proof}
We note that if $\eta=a(\xi-\mu_\ell f)+bf$ is a big rational class on $\mathbb{P}(E)$ and the induced linear flag $Y_\bullet$ belongs to the Schubert cell $\Omega_w$ with respect to a reference flag $Y_\bullet^{\HN}$, then Theorem \ref{theo: okounkov bodies} gives
$$\Delta_{Y_\bullet}(\xi-\mu_1f)=\left\{(\nu_1,\ldots,\nu_r)\in \mathbb{R}_{\geq 0}^r \;|\;0\leq \nu_1\leq \mu_\ell-\mu_1,\;(\nu_2,\ldots,\nu_r)\in \square_{\mu_1+\nu_1}^{w} \right\} $$
and hence
$$a\Delta_{Y_\bullet}(\xi-\mu_1f)=\left\{(\nu_1,\ldots,\nu_r)\in \mathbb{R}_{\geq 0}^r \;|\;0\leq \nu_1\leq a(\mu_\ell-\mu_1),\;(\nu_2,\ldots,\nu_r)\in a\square_{\mu_1+\frac{1}{a}\nu_1}^{w} \right\}.$$
We compute that for $t^*=b-a(\mu_\ell-\mu_1)$ we have
$$a\Delta_{Y_\bullet}(\xi-\mu_1f)+t^*\vec{e}_1=\left\{(\nu_1,\ldots,\nu_r)\in \mathbb{R}_{\geq 0}^r \;|\;0\leq \nu_1\leq b,\;(\nu_2,\ldots,\nu_r)\in a\square_{\mu_1+\frac{1}{a}(\nu_1-t^*)}^{w} \right\}$$
with $\mu_1+\frac{1}{a}(\nu_1-t^*)=\mu_\ell - \frac{1}{a}(b-\nu_1)$. The result follows from Lemma \ref{lemma:reduction to big nef}. \end{proof}

{\bibliography{bibliokounkov}{}}

\end{document}